\documentclass{amsart}%
\usepackage{amsmath,amsthm}
\usepackage{amsfonts}
\usepackage{amssymb}
\usepackage{graphicx}
\usepackage[colorlinks=true,citecolor=red]{hyperref}
\usepackage[usenames,dvipsnames]{color}
\usepackage{amsmath}%
\providecommand{\U}[1]{\protect\rule{.1in}{.1in}}
\providecommand{\U}[1]{\protect\rule{.1in}{.1in}}
\providecommand{\U}[1]{\protect\rule{.1in}{.1in}}
\newtheorem{theorem}{Theorem}[section]
\newtheorem{corollary}[theorem]{Corollary}
\newtheorem{lemma}[theorem]{Lemma}
\newtheorem{proposition}[theorem]{Proposition}
\theoremstyle{definition}
\newtheorem{remark}[theorem]{Remark}

\begin{document}
\title{Finite axiomatizability for profinite groups}
\author{Andre Nies, Dan Segal and Katrin Tent}
\thanks{The first author was 
partially supported by the Marsden fund of New Zealand, 
13-UOA-184 and 19-UOA-346.}
\begin{abstract}
A group is finitely axiomatizable (FA) in a class $\mathcal{C}$ if it can be
determined up to isomorphism within $\mathcal{C}$ by a sentence in the
first-order language of group theory. We show that profinite groups of various
kinds are FA in the class of profinite groups, or in the class of pro-$p$
groups for some prime~$p$. Both algebraic and model-theoretic methods are
developed for the purpose. Reasons why certain groups cannot be FA are also discussed.

\end{abstract}
\maketitle

\setcounter{tocdepth}{1}

\tableofcontents

\section{Introduction}

Some properties of a group can be expressed by a sentence in the first-order
language $L_{\mathrm{gp}}$ of group theory, and some cannot. If the group is
assumed to be \emph{finite}, a lot more can be said about it in first-order
language than in the general case. We mention examples of these phenomena below.

The strongest property of a group $G$ is that of `being isomorphic to $G$'. If
this can be expressed by a first-order sentence, $G$ is said to be
\emph{finitely axiomatizable}, henceforth abbreviated to \emph{FA.} It is
obvious that every finite group is FA: if $\left\vert G\right\vert =n$, the
fact that $G$ has exactly $n$ elements and that they satisfy the
multiplication table of $G$ is clearly a first-order property. An infinite
group cannot be FA by the L\"{o}wenheim-Skolem Theorem (\cite{TZ}, Thm 2.3.1);
to make the question interesting we have to limit the universe of groups under
consideration. For example, the first author in \cite{NSG} called a finitely
generated, infinite group $G$ \emph{QFA} (for quasi-finitely axiomatizable) if
some first order sentence determines it up to isomorphism within the class of
finitely generated groups (this is \emph{not} the notion of `quasi-finite
axiomatizability' used in model theory, cf. \cite{P}, Chapt. 3).

He showed that several well-known groups, such as the restricted wreath
product $C_{p}\wr\mathbb{Z}$, have this property (here $C_{p}$ denotes the
cyclic group of order $p$). The QFA nilpotent groups are completely
characterized by Oger and Sabbagh in \cite{OS}. Further results were obtained
by Lasserre \cite{L}. Nies~\cite{NDG} contains a survey up to 2007.

In the present paper, we address the question of relative finite
axiomatizability in the universe of \emph{profinite groups}. These (unless
finite) are necessarily uncountable, so cannot be finitely generated as
groups; but from some points of view they behave rather like finite groups.
For example, Jarden and Lubotzky show in \cite{JL} that if $G$ is a
(topologically) finitely generated profinite group, then the elementary theory
of $G$ characterizes $G$ up to isomorphism among all profinite groups (cf.
\cite{SW}, Thm. 4.2.3); in this case one says that $G$ is
\emph{quasi-axiomatizable}. This is very different from the situation in
abstract groups: for example, a celebrated theorem of Sela \cite{S} (see also
\cite{KM}) shows that all finitely generated non-abelian free groups have the
same elementary theory.

The elementary theory of $G$ consists of \emph{all} the sentences satisfied by
$G$. We consider the question: which profinite groups can be characterized by
a single sentence? To make this more precise, let us say that a group $G$ is
\emph{FA (wrt} $L$\emph{) in} $\mathcal{C}$ if $\mathcal{C}$ is a class of
groups containing $G$, $L$ is a language, and there is a sentence $\sigma_{G}$
of $L$ such that for any group $H$ in $\mathcal{C}$,%
\[
H\models\sigma_{G}\text{ if and only if }H\cong G.
\]
For instance, QFA means: FA (wrt $L_{\mathrm{gp}}$) in the class of all f.g.
groups. When $\mathcal{C}$ is a class of profinite groups, isomorphisms are
required to be \emph{topological}. \ Usually, we will write `FA' to mean `FA
in the class of all profinite groups.'

\subsection{Classes of groups and their theories}

It is often the case that a natural class of (abstract) groups cannot be
axiomatized in the first-order language $L_{\mathrm{gp}}$ of group theory.
This holds for the class of simple groups (see \cite{JSW}), the $2$-generated
groups, the finitely generated groups, and classes such as nilpotent or
soluble groups, none of which is closed under the formation of ultraproducts.

Since finite groups are FA, \emph{every} class $\mathcal{C}$ of finite groups
can be axiomatized within the finite groups: a finite group $H$ is in
$\mathcal{C}$ if and only if $H\models\lnot\sigma_{G}$ for every finite group
$G\notin\mathcal{C}$ (cf. \cite{JSW}, \S 1). Whether such a class can be
\emph{finitely} axiomatized within the finite groups is usually a much subtler
question. For example, a theorem of Felgner shows that this holds for the
class of non-abelian finite simple groups (see \cite{JSW}, Theorem 5.1), and
Wilson \cite{WFS} shows that the same is true for the class of finite soluble
groups. On the other hand, Cornulier and Wilson show in \cite{CW} that
nilpotency cannot be characterized by a first-order sentence in the class of
finite groups.

The main object of study in Nies \cite{NSG} was the first-order separation of
classes of groups $\mathcal{C}\subset\mathcal{D}$. Even if the classes are not
axiomatizable, can we distinguish them using first-order logic, by showing
that some sentence $\phi$ holds in all groups of $\mathcal{C}$ but fails in
some group in $\mathcal{D}$? If this holds, one says that $\mathcal{C}$ and
$\mathcal{D}$ are \emph{first-order separated}. One way to establish this is
to find a witness for separation: a group $G$ not in $\mathcal{C}$ that is FA
in~$\mathcal{D}$. Then one takes $\phi$ to be the negation of a sentence
describing~$G$ within~$\mathcal{D}$.

Some of our results serve to provide first-order separations of interesting
classes of profinite groups:

\begin{itemize}
\item the finite rank profinite groups are first-order separated from the
(topologically) finitely generated profinite groups by Prop.\ \ref{wr}

\item similarly for pro-$p$ groups, also by Prop.\ \ref{wr}

\item the f.g.\ profinite groups are first-order separated from the class of
all profinite groups by Cor.\ \ref{sep}.
\end{itemize}

\subsection{Obstructions to finite axiomatizability}

We know of two obstructions to being FA for a profinite group: the centre may
`stick out too much', or the group may involve too many primes. The first is
exemplified by the following result of Oger and Sabbagh, which generalizes
work of Wanda Szmielew (see \cite{HMT}, Thm A.2.7) for infinite abelian
groups; here $\mathrm{Z}(G)$ denotes the centre of $G$ and $\Delta
(G)/G^{\prime}$ the torsion subgroup of $G/G^{\prime}$ where $G^{\prime}$ is
the derived group:

\begin{theorem}
\label{neg1}\emph{(\cite{OS}, Theorem 2).} Let $G$ be a\ group such that
$\mathrm{Z}(G)\nsubseteq\Delta(G)$. If $\phi$ is a sentence such that
$G\models\phi$, then $G\times C_{p}\models\phi$ for almost all primes $p$.
\end{theorem}

\noindent If for example $G$ is a finitely generated profinite group, then
$G\times C_{p}\ncong G$ for every prime $p$, so $G$ cannot be FA.

The second obstruction comes from a different direction. The Feferman-Vaught
Theorem controls first-order properties of Cartesian products, and in
\S \ref{negsec} we note its consequence

\begin{proposition}
\bigskip\label{products}Let $\left(  G_{i}\right)  _{i\in I}$ be an infinite
family of groups. If $\phi$ is a sentence such that $\prod_{i\in I}%
G_{i}\models\phi$, then there exist $q\neq r\in I$ such that%
\[
G_{r}\times\prod_{i\in I\smallsetminus\{q\}}G_{i}\models\phi.
\]

\end{proposition}

We will deduce

\begin{proposition}
\label{negprod}\emph{(i)} If the pronilpotent group $G$ is FA in the class of
profinite, or pronilpotent, groups, then the set of primes $p$ such that $G$
has a nontrivial Sylow pro-$p$ subgroup is finite.

\emph{(ii)} Let $\mathfrak{G}$ be a linear algebraic group of positive
dimension defined over $\mathbb{Q}$.
If $\mathfrak{G}(\mathbb{Z}_{\pi})$ is FA in the class of profinite groups
then the set of primes $\pi$ is finite.
\end{proposition}

\noindent Here $\mathbb{Z}_{\pi}=\prod_{p\in\pi}\mathbb{Z}_{p}.$

\medskip

Our main results tend to suggest that for a wide range of profinite groups
these are the only obstructions. For example, the converse of Proposition
\ref{negprod}(ii) holds when $\mathfrak{G}$ is a simple Chevalley group; for
$\mathfrak{G=}$ $\mathrm{PSL}_{n}$ this is a consequence of Theorem
\ref{slsl}, for other cases see \cite{ST}, Theorem 1.5 (and Remark(ii),
\S \ref{profringsec} below).

However, there are two caveats.

\emph{One}: it is obvious that two groups that are isomorphic (as abstract
groups) must satisfy the same first-order sentences; it is possible for
non-isomorphic profinite groups to be isomorphic as abstract groups (cf.
\cite{K1}), and such groups cannot be FA as profinite groups. In general,
there is a strict hierarchy of implications for a profinite group $G$:

\begin{itemize}
\item $G$ \emph{is FA} $\ \Longrightarrow~G$ \emph{is quasi-axiomatizable}
$\Longrightarrow$ $G$ \emph{is `algebraically rigid',}
\end{itemize}

\noindent the third condition meaning:\ \emph{any profinite group abstractly
isomorphic to} $G$ \emph{is topologically isomorphic to} $G$.

The problem does not arise for groups that are `strongly complete': this means
that \emph{every subgroup of finite index is open}. Every group homomorphism
from such a group to any profinite group is continuous; in fact these groups
are also quasi-axiomatizable (see \cite{H}). Every finitely generated
profinite group is strongly complete (see Theorem \ref{NSSC} below). Most of
the profinite groups we consider in this paper are finitely generated (as
topological groups), but not all (see Cor. \ref{sep}).

\emph{Two: } There are only countably many sentences, but uncountably many
groups, even among those that avoid the above obstructions. We exhibit in
\S \ref{negsec} a family of such pro-$p$ groups parametrized by the $p$-adic integers.

There are various ways around this problem. One may restrict attention to the
groups that have a \emph{strictly finite presentation}: a profinite (or
pro-$p)$ group $G$ has this property if it has a finite presentation as a
profinite (or pro-$p)$ group in which the relators are finite group words;
equivalently, if $G$ is the completion of a finitely presented abstract group.
In \S \ref{present} we define a more general concept called $L$%
\emph{-presentation}, which allows for groups like $C_{p}\widehat{\wr
}\,\mathbb{Z}_{p}$, the pro-$p$ completion of the aforementioned $C_{p}%
\wr\mathbb{Z}$: this is not strictly finitely presentable, but it is finitely
presented within the class of metabelian pro-$p$ groups (cf. \cite{Ha} for
abstract metabelian groups); Proposition \ref{wr} shows that it is FA in the
class of all profinite groups. An $L$-presentation is like a finite
presentation in which the usual relations may be replaced by any sentence in
the language $L$.

Another way is to enlarge the first-order language: given a finite set of
primes~$\pi$, we take $L_{\pi}$ to be the language $L_{\mathrm{gp}}$ augmented
with extra unary function symbols~$P_{\lambda},$ one for each $\lambda
\in\mathbb{Z}_{\pi}=\prod_{p\in\pi}\mathbb{Z}_{p}$; for a group element $g$,
$P_{\lambda}(g)$ is interpreted as the profinite power $g^{\lambda}$. We shall
see that many pro-$p$ groups are indeed FA (wrt $L_{\{p\}}$) within the class
of pro-$p$ groups.

\subsection{Bi-interpretation}

We shall explore two different ways of showing that profinite groups are FA.
The first is a model-theoretic procedure known as \emph{bi-interpretation},
first used to show that certain groups are finitely axiomatizable by Khelif
\cite{Kh}. This is defined e.g. in \cite{P}, Def. 3.1
; further applications of bi-intepretation are described in \cite{NDG}, \S 7.7
and in \cite{AKS}, \S 2.

If (a group) $A$ is interpreted in (a ring) $B$ and $B$ is interpreted in $A$,
we have an `avatar' $\widetilde{A}$ of $A$ in some $B^{(n)}$, and an avatar
$\widetilde{B}$ of $B$ in some $A^{(m)}$. Composing these procedures produces
another avatar $\widetilde{\widetilde{A}}$ of $A$ in $A^{(mn)}$, and an
isomorphism from $A$ to $\widetilde{\widetilde{A}}$. Similarly, one obtains an
isomorphism from $B$ to $\widetilde{\widetilde{B}}\subseteq B^{(nm)}$. If
these two isomorphisms are definable (in the language of groups, respectively
rings), then $A$ and $B$ are said to be \emph{bi-interpretable}.

When dealing with profinite groups and profinite rings, the definition has to
be tweaked to take account of the topology. We postpone the precise
definitions to \S \textbf{\ref{bisecgr}}, where the following result is established:

\begin{itemize}
\item \emph{Let }$R$\emph{\ be a profinite ring and }$G$ \emph{\ a profinite
group. If } $G$\emph{\ is `topologically bi-interpretable' with }%
$R$\emph{\ then }$G$ \emph{is FA in profinite groups iff} $R$\emph{\ is FA in
profinite rings, assuming algebraic rigidity where appropriate.}
\end{itemize}

As an illustration of the method, we prove

\begin{theorem}
\label{sl2rt}Let $R$ be a complete, unramified regular local ring with finite
residue field $\kappa$. Then each of the profinite groups $\mathrm{Af}_{1}%
(R)$, $\mathrm{SL}_{2}(R)$ is FA in the class of profinite groups, assuming in
the second case that $\mathrm{char}(\kappa)$ is odd.
\end{theorem}

\noindent Here $\mathrm{Af}_{1}(R)=(R,+)\rtimes R^{\ast}$ denotes the
$1$-dimensional affine group over $R$. The theorem combines Theorems
\ref{regloc}, \ref{affine} and \ref{sl2_biint copy(1)}, proved below. This
result is extended in \cite{ST} to Chevalley groups of rank at least $2$ over
a more general class of rings.

Although we do not pursue this aspect, it may be of interest to mention that
the proof of Theorem \ref{sl2rt} can be adapted to show that the respective
groups are uniformly bi-intepretable with the corresponding rings, i.e. the
defining formulae are independent of the ring.

\medskip In Theorem~\ref{sl2rt} the rings in question are the following:

\begin{itemize}
\item power series rings in finitely many variables over a finite field

\item power series rings in finitely many variables over an unramified
$p$-adic ring $\mathbb{Z}_{p}[\zeta]$ ($\zeta$ a $(p^{f}-1)$th root of unity).
\end{itemize}

While the groups $\mathrm{SL}_{2}(R)$ (for these rings $R$) are finitely
generated as profinite groups (see Prop.\ \ref{linsc} below), the groups
$\mathrm{Af}_{1}(R)$ are \emph{not}, in most cases (see the remark~following
the proof of Prop. \ref{aff}); this shows that a profinite group can be very
far from finitely presented and still be FA. It also establishes

\begin{corollary}
\label{sep}The classes of f.g. profinite groups and all profinite groups are
first-order separable, with witness group $\mathrm{Af}_{1}(\mathbb{F}%
_{p}[[t]])$.
\end{corollary}

\subsection{$p$-adic analytic groups, and more}

The other approach to establishing that certain groups are FA is purely
group-theoretic; as such, it is limited to groups that are `not very big', in
a sense about to be clarified. A pro-$p$ group is an inverse limit of finite
$p$-groups, where by convention $p$ always denotes a prime. We observed above
that `involving too many primes' can be an obstruction to being FA. In fact
all our positive results concern groups that are virtually pro-$p$ (that is,
pro-$p$ up to finite index), or finite products of such groups.

The pro-$p$ groups in question are compact $p$-\emph{adic analytic groups}.
This much-studied class of groups can alternatively be characterized as the
\emph{virtually pro-}$p$ \emph{groups of} \emph{finite rank}; the profinite
group $G$ has finite \emph{rank} $r$ if every closed subgroup can be generated
by $r$ elements (`generated' will always mean: `generated topologically'). For
all this, see the book \cite{DDMS}, in particular Chapter 8.

The possibility of showing that (some of) these groups are FA rests on the
fact that they have a finite \emph{dimension}: this can be used rather like
the order of a finite group, to control when a group has no proper quotients
of the same `size'.

Let $\pi=\{p_{1},\ldots,p_{k}\}$ be a finite set of primes. A $\mathcal{C}%
_{\pi}$ \emph{group }is one of the form $G_{1}\times\cdots\times G_{k}$ where
$G_{i}$ is a pro-$p_{i}$ group for each $i$. A $\mathcal{C}_{\pi}$ group of
finite rank need not be strictly finitely presented, but it always has an
$L_{\pi}$ presentation (see Subection \ref{present}).

The first main result about $\mathcal{C}_{\pi}$ applies in particular to all
$p$-adic analytic pro-$p$ groups, but limits the universe:

\begin{theorem}
\label{frank}Every $\mathcal{C}_{\pi}$ group of finite rank is FA (wrt
$L_{\pi}$) in the class $\mathcal{C}_{\pi}$; if it has an $L_{\mathrm{gp}}%
$-presentation (e.g. if it is strictly finitely presented) then it is FA (wrt
$L_{\mathrm{gp}}$) in the class $\mathcal{C}_{\pi}$.
\end{theorem}

This will be the key to several theorems showing that groups in certain
limited classes of $\mathcal{C}_{\pi}$ groups are FA among all profinite
groups. The first of these is a profinite analogue of \cite{OS}, Theorem 10.

\begin{theorem}
\label{nilp} Let $G$ be a nilpotent $\mathcal{C}_{\pi}$ group, and suppose
that $G$ has an $L_{\mathrm{gp}}$-presentation. Then $G$ is FA in the class of
all profinite groups if and only if ${\mathrm{Z}(G)\subseteq\Delta(G)}$.
\end{theorem}

\noindent(The hypothesis implies that $G$ is f.g.; as a product of finitely
many nilpotent pro-$p$ groups, $G$ then has finite rank.)

Note that by Proposition \ref{negprod}(i), both results would fail if $\pi$
were an infinite set of primes.

\medskip

In \S \ref{pronilpsec} we establish some results intermediate between the last
two, characterizing those $\mathcal{C}_{\pi}$ groups of finite rank that are
FA in the class of all \emph{pronilpotent} groups.

The final main result shows how these methods may be applied to $p$-adic
analytic groups that are far from nilpotent.

\begin{theorem}
\label{slsl}Let $n\geq2$ and let $p$ be an odd prime such that $p\nmid n$.
Then each of the groups%
\[
\mathrm{SL}_{n}^{1}(\mathbb{Z}_{p}),~\mathrm{SL}_{n}(\mathbb{Z}_{p}%
),~\mathrm{PSL}_{n}(\mathbb{Z}_{p})
\]
is FA in the class of profinite groups.
\end{theorem}

\noindent Here $\mathrm{SL}_{n}^{1}(\mathbb{Z}_{p})$ denotes the principal
congruence subgroup modulo $p$ in $\mathrm{SL}_{n}(\mathbb{Z}_{p})$.

The proof for $\mathrm{SL}_{n}^{1}(\mathbb{Z}_{p})$ uses both Theorem
\ref{frank} and Theorem \ref{nilp}, which can be applied to the upper
unitriangular group (when $n\geq3$). The extension to $\mathrm{SL}%
_{n}(\mathbb{Z}_{p})$ depends on Theorem \ref{FinExt}, proved in
\S \ref{finite_ext}, which establishes some sufficient conditions for a finite
extension of an FA group to be FA.

A different proof for $\mathrm{SL}_{n}(\mathbb{Z}_{p})$ and $\mathrm{PSL}%
_{n}(\mathbb{Z}_{p})$ with $n\geq3$, via bi-interpretability, appears in
\cite{ST}.

\subsection{Some problems}

The Oger-Sabbagh theorem characterizing the nilpotent (abstract) groups that
are QFA has been extended to polycyclic groups by Lasserre \cite{L}: such a
group $G$ is QFA iff ${\mathrm{Z}(H)\subseteq\Delta(H)}$ for each subgroup $H$
of finite index. The analogous class of pro-$p$ groups is the soluble pro-$p$
groups of finite rank, suggesting

\medskip\noindent\textbf{Problem 1.} Let $G$ be a soluble pro-$p$ group of
finite rank. Show that the following are equivalent:

\begin{description}
\item[a] $G$ \emph{is FA in the class of profinite groups}

\item[b] ${\mathrm{Z}(H)\subseteq\Delta(H)}$ \emph{for each open subgroup} $H$
\emph{of} $G$.
\end{description}

\medskip We do not know whether the hypothesis of finite rank in
Theorem~\ref{frank} is necessary, even if $\pi$ only contains a single prime.

\medskip\noindent\textbf{Problem 2.} (a) Is every finitely generated pro-$p$
group $G$ FA wrt $L_{\{p\}}$ in $\mathcal{C}_{p}$? \break(b) If $G$ is
strictly finitely presented as a pro-$p$ group, is $G$ FA (wrt $L_{\mathrm{gp}%
}$) in the class $\mathcal{C}_{p}$?

\medskip The answer is probably `no', to both. A possible candidate for a
counterexample is a non-abelian free pro-$p$ group $\widehat{\left(
F_{n}\right)  }_{p},$ $n\geq2$. As far as we know, even the following is open:

\medskip\noindent\textbf{Problem 3. }Is the group $\widehat{\left(
F_{n}\right)  }_{p}$ (where $n\geq2$) FA in the class of profinite groups?

\subsection{Organization of the paper}

The next section introduces notation and presents some general results about
definability in profinite groups. Section \ref{finite_ext} is devoted to
showing that under certain conditions, a finite extension of an FA group is
again FA; this is useful in situations like that of Theorem \ref{slsl}, which
deal with groups that are virtually pro-$p$ but not actually pro-$p$. Section
\ref{bisecgr} deals with bi-interpretability and applications. The material
about $\mathcal{C}_{\pi}$-groups occupies Sections \ref{frcpisec} and
\ref{slnsec}. Some negative results are collected together in
Section~\ref{negsec}. The short Section~\ref{s:listform} consists of a list of
first-order formulas for lookup.

\bigskip We thank the painstaking referee for some corrections and helpful suggestions.

\section{Definable subgroups}

\label{defsub}

For a group $G$ and a formula $\kappa(x)$ (possibly with parameters
$\overline{g}$ from $G$), we write%
\[
\kappa(G)=\kappa(\overline{g};G):=\{x\in G\mid G\models\kappa(\overline
{g},x)\}.
\]
(The notation will also be used, \emph{mutatis mutandis}, for rings.) A
subgroup is \emph{definable }if it is of this form; unless otherwise stated,
$\kappa$ is supposed to be a\ formula of $L_{\mathrm{gp}}$. Note that
$\kappa(G)$ is a subgroup iff $G\models\mathrm{s}(\kappa)$ where%
\[
\mathrm{s}(\kappa)\equiv\exists x.\kappa(x)\wedge\forall x,y.\left(
\kappa(x)\wedge\kappa(y)\rightarrow\kappa(x^{-1}y)\right)  ,
\]
and $\kappa(G)$ is a normal subgroup iff $G\models\mathrm{s}_{\vartriangleleft
}(\kappa)$ where%
\[
\mathrm{s}_{\vartriangleleft}(\kappa)\equiv\mathrm{s}(\kappa)\wedge\forall
x,y.\left(  \kappa(x)\rightarrow\kappa(y^{-1}xy)\right)  .
\]

We will say that a subgroup $H$ is \emph{definably closed} if $H=$ $\kappa(G)
$ for a formula $\kappa$ such that in any profinite group $M$, the subset
$\kappa(M)$ is necessarily closed.

Suppose that $H=\kappa(G)$ is a definable subgroup of $G$. By the usual
relativization process, for any formula $\varphi(y_{1},\ldots,y_{k})$ there is
a `restriction' formula $\mathrm{res}(\kappa,\varphi)(y_{1},\ldots,y_{k})$
such that for each $k$-tuple $\overline{b}\in H^{(k)}$ we have%
\[
G\models\mathrm{res}(\kappa,\varphi)(\overline{b})~\Longleftrightarrow
~H\models\varphi(\overline{b}).
\]
(Note that $\mathrm{res}(\kappa,\varphi)$ is obtained from $\varphi$ by
relativizing the quantifiers of $\varphi$, i.e. replacing any expression
$\forall z\psi(z)$ by $\forall z.(\kappa(z)\longrightarrow\psi(z)),$ and any
expression $\exists z\psi(z)$ by $\exists z.(\kappa(z)\wedge\psi(z))$.
Clearly, if $\varphi$ is quantifier-free, then $\mathrm{res}(\kappa,\varphi)$
is just $\varphi$. )

Similarly, if $N=\kappa(G)$ is a definable \emph{normal} subgroup, there is a
`lifted' formula $\mathrm{lift}(\kappa,\varphi)$ such that%
\[
G\models\mathrm{lift}(\kappa,\varphi)(\overline{b})~\Longleftrightarrow
~G/N\models\varphi(\widetilde{b_{1}},\ldots,\widetilde{b_{k}}),
\]
where $\tilde{b}$ denotes the image of $b$ modulo $N$. To obtain
$\mathrm{lift}(\kappa,\varphi)$ we replace each atomic formula $x=y$ in
$\varphi$ with $\kappa(x^{-1}y)$.

Suppose that $\kappa(G)$ is a definable subgroup, and let $n\in\mathbb{N}$.
Then
\begin{align*}
\left\vert G:\kappa(G)\right\vert  &  \leq n\Longleftrightarrow G\models
\mathrm{ind}(\kappa;n),\\
\left\vert G:\kappa(G)\right\vert  &  =n\Longleftrightarrow G\models
\mathrm{ind}(\kappa;n)\wedge\lnot\mathrm{ind}(\kappa;n-1):=\mathrm{ind}^{\ast
}(\kappa;n),
\end{align*}
where%
\[
\mathrm{ind}(\kappa;n)\equiv\exists u_{1},\ldots,u_{n}.\forall x.\bigvee
_{j}\kappa(x^{-1}u_{j}).
\]
We define the frequently used formula%
\[
\mathrm{com}(x,y):=(xy=yx)
\]

For a profinite group $G$ and $X\subseteq G,$ the closure of $X$ is denoted
$\overline{X}.$ (This is not to be confused with $\overline{x}$, which stands
for a tuple $(x_{1} , \ldots, x_{n})$ .)

We write $X\leq_{c}G,$ resp. $X\leq_{o}G,$ for `$X$ is a closed, resp. open,
subgroup of $G$'.

For any group $G$ (abstract or profinite) and $Y\subseteq G,$ the subgroup
generated (algebraically) by $Y$ is denoted $\left\langle Y\right\rangle $.
For $q\in\mathbb{N}$, $G^{\{q\}}=\{g^{q}\mid g\in G\}$ is the set of $q$-th
powers and $G^{q}=\left\langle G^{\{q\}}\right\rangle .$

The derived group of $G$ is $G^{\prime}=\left\langle [x,y]\mid x,y\in
G\right\rangle $. Note that%
\[
G^{\prime}G^{q}=G^{\prime}G^{\{q\}}.
\]

The key fact that makes f.g.\ profinite groups accessible to first-order logic
is the \emph{definability of open subgroups.}\ We shall use the following
without special mention:

\begin{theorem}
\label{NSSC}\emph{(Nikolov and Segal)} Let $G$ be a f.g.\ profinite group.

\emph{(i) }Every subgroup of finite index in $G$ is both open and definably
closed (with parameters).

\emph{(ii) }Each term\emph{\ }$\gamma_{n}(G)$ of the lower central series of
$G$ is closed and definable (without parameters).

\emph{(iii) }Every group homomorphism from\emph{\ }$G$ to a profinite group is continuous.
\end{theorem}

\begin{proof}
The definability of subgroups in a profinite group is related to the topology
of the group through the concept of verbal width. A word $w$ has \emph{width}
$f$ in a group $G$ if every product of $w$-values or their inverses is equal
to such a product of length $f$. The verbal subgroup $w(G)$ generated by all
$w$-values is closed in $G$ if and only if $w$ has finite width (\cite{SW},
Prop.\ 4.1.2); in this case it is definable, by the formula $\kappa_{w,f}(x)$
which expresses that%
\begin{equation}
x\in G_{w}\cdot\ldots\cdot G_{w}\text{ \ \ (}f\text{ factors)} \label{vsub}%
\end{equation}
where $w=w(x_{1},\ldots,x_{k})$ has width $f$ and $G_{w}=\{w(\overline
{g})^{\pm1}~\mid~\overline{g}\in G^{(k)}\}$. This formula defines a closed
subset in every profinite group, since the verbal mapping $G^{(k)}\rightarrow
G$ defined by $w$ is continuous, hence has compact image.

In a finitely generated profinite group, each lower-central word and all power
words have finite width (\cite{NS1}, \cite{NS2}, \cite{NS}). (ii) follows at once.

For (i), suppose $H$ is a subgroup of finite index in $G$. Then $H\geq
G^{q}=\left\langle g^{q}\mid g\in G\right\rangle $ for some $q$, and $G^{q}$
is definably closed by the preceding remarks, because the word $x^{q}$ has
finite width. If $G^{q}\leq N\vartriangleleft_{o}G$ then $G/N$ is a finite
$d=\mathrm{d}(G)$-generator group of exponent dividing $q$, hence has order
bounded by a finite number $\beta(d,q)$ (by the positive solution of the
Restricted Burnside Problem \cite{Z}, \cite{Z2}). As $G^{q}$ is the
intersection of all such $N$ it follows that $G^{q}$ is open. Now (i) follows
by the lemma below.

(iii) is an easy consequence of the fact that every subgroup of finite index open.
\end{proof}

\begin{lemma}
\label{findef}Suppose $N$ is a definable subgroup in a group $G$. If $N\leq
H\leq G$ and $\left\vert H:N\right\vert $ is finite then $H$ is definable, by
a formula with parameters. If $N$ is definably closed then so is $H$. If
$G=N\left\langle X\right\rangle $ for some subset $X,$ we may choose the
parameters in $X$.
\end{lemma}

\begin{proof}
Say $N=\kappa(G)$ and $H=Ng_{1}\cup\ldots\cup Ng_{n}$. Then $H$ is defined by
$\bigvee_{i=1}^{n}\kappa(xg_{i}^{-1})$. The second claim is clear since the
union of finitely many translates of a closed set is closed. For the final
claim, we may replace each $g_{i}$ by a suitable word on $X$.
\end{proof}

\medskip

\textbf{Remark.} If every subgroup of finite index in $G$ is open, then every
subgroup of finite index contains a definable open subgroup, whether or not
$G$ is f.g.: this follows from \cite{WS}, Theorem 2 in a similar way to the
proof of (ii) above; it is implicit in the proof of \cite{H}, Theorem 3.11.

\medskip

The special case of these results where $G$ is a pro-$p$ group is much easier,
and suffices for most of our applications; see e.g. \cite{DDMS}, Chapter 1,
ex. 19 and \cite{SW}, \S 4.3.

Note that subgroups like $w(G)$ when $w$ is a word of finite width are
definable as in (\ref{vsub}) without parameters.

When proving that a certain group $G$ is FA in some class $\mathcal{C}$, we
often establish a stronger property, namely: for some finite (usually
generating) tuple $\overline{g}$ in $G$, there is a formula $\sigma_{G}$ such
that for a group $H$ in $\mathcal{C}$ and a tuple $\overline{h}$ in $H$,
$H\models\sigma_{G}(\overline{h})$ if and only if there is an isomorphism from
$G$ to $H$ mapping $\overline{g}$ to $\overline{h}$, a situation denoted by
$(G,\overline{g})\cong(H,\overline{h})$. In this case we say that
$(G,\overline{g})$ is FA in $\mathcal{C}$. Of course, this implies that $G$ is
FA in $\mathcal{C}$: indeed, for $H$ in $\mathcal{C}$, we have $H\cong G$ if
and only if $H\models\exists\overline{x}.\sigma_{G}(\overline{x})$.

\section{Finite extensions\label{finite_ext}}

If a group $G$ is FA, one would expect that (definable) subgroups of finite
index in $G$ and finite extension groups of $G$ should inherit this property.
In this section we establish the latter under some natural hypotheses.

Fix a class $\mathcal{C}$ of profinite groups, and assume that $\mathcal{C}$
is closed under taking open subgroups. $L\supseteq L_{\mathrm{gp}}$ is a
language. By `FA' we mean FA (wrt $L$) in $\mathcal{C}$.

Given a group $N$ and elements $h_{1},\ldots,h_{s}\in N$, we say that an
element $g$ of $N$ is $\overline{h}$\emph{-definable} in $N$ if there is a
formula $\phi_{g}$ such that for $c\in N$,
\begin{equation}
N\models\phi_{g}(\overline{h},c)\Longleftrightarrow c=g. \label{defelt}%
\end{equation}
This holds in particular if $g\in\left\langle h_{1},\ldots,h_{s}\right\rangle
$.

\medskip

\textbf{Remarks.} (i) if $\theta:N\rightarrow M$ is an isomorphism and
(\ref{defelt}) holds, then $g\theta$ is the unique element $b$ of $M$ such
that $N\models\phi_{g}(\overline{h}\theta,b).$

(ii) If $(N,\overline{h)}$ is FA and $g$ is $\overline{h}$-definable in $N$
then $(N,(\overline{h},g))$ is FA.

(iii) If $g$ is $\overline{h}$-definable in $N$ and $N=\kappa(G)$ is a
definable subgroup of $G,$ then $g$ is $\overline{h}$-definable in $G,$ by the
formula
\[
\kappa(y)\wedge\mathrm{res}(\kappa,\phi_{g}).
\]

\begin{theorem}
\label{FinExt}Let
\[
N=\overline{\left\langle h_{1},\ldots,h_{s}\right\rangle }\vartriangleleft
_{o}G=\overline{\left\langle g_{1},\ldots,g_{r}\right\rangle }\in\mathcal{C},
\]
and assume that $(N,\overline{h})$ is FA. Then $G$ is FA provided one of the
following holds:

\emph{(a) }$N\cap\left\langle g_{1},\ldots,g_{r}\right\rangle =\left\langle
h_{1},\ldots,h_{s}\right\rangle $, in which case $(G,\overline{g})$ is FA; or

\emph{(b) }$\mathrm{Z}(N)=1$, $\{h_{1},\ldots,h_{s}\}\subseteq\left\langle
g_{1},\ldots,g_{r}\right\rangle $, and $h_{i}^{g_{j}}$ is $\overline{h}%
$-definable in $N$ for each $i$ and $j$.
\end{theorem}

\begin{proof}
Say $\left\vert G:N\right\vert =m$. By Theorem \ref{NSSC} there is a formula
$\kappa$ such that $N=\kappa(G;\overline{g})$, and such that $\kappa$ always
defines a closed subset in any profinite group. Thus $G$ satisfies%
\[
\Phi_{1}(\overline{g}):=\mathrm{s}_{\vartriangleleft}(\kappa(\overline
{g}))\wedge\mathrm{ind}^{\ast}(\kappa(\overline{g}),m),
\]
which asserts that $\kappa(G;\overline{g})$ is a closed normal subgroup of
index $m$ (and is therefore open).

By hypothesis, there is a formula $\psi$, where $N\models\psi(\overline{h})$,
such that if $k_{1},\ldots,k_{s}\in M\in\mathcal{C}$ and $M\models
\psi(\overline{k})$ then there is an isomorphism $N\rightarrow M$ sending
$\overline{h}$ to $\overline{k}$. For each $i$ there is a word $w_{i}$ such
that $h_{i}=w_{i}(\overline{g})$; then $G$ satisfies%
\begin{equation}
\Phi_{2}(\overline{g}):=\bigwedge_{i=1}^{s}\kappa(\overline{g},w_{i}%
)\wedge\mathrm{res}(\kappa(\overline{g}),\psi(w_{1},\ldots,w_{s})),
\label{defsubn}%
\end{equation}
where for aesthetic reasons $w_{i}$ is written in place of $w_{i}(\overline
{g}),$ a convention we keep throughout this proof.

Since $\mathcal{C}$ is closed under taking open subgroups, $\Phi_{1}%
(\overline{g})$ implies that $\kappa(G;\overline{g})\in\mathcal{C}$, and then
$\Phi_{2}(\overline{g})$ ensures that $\kappa(G;\overline{g})\cong$ $N$. We
set%
\[
\Phi:=\Phi_{1}\wedge\Phi_{2}.
\]

To fix the isomorphism type of $G,$ we need also to specify the conjugation
action of $G$ on $N$, the quotient $G/N,$ and the extension class. These are
done in the following manner. To begin with, note that $G=N\left\langle
g_{1},\ldots,g_{r}\right\rangle $ because $N$ is open; hence there exists a
transversal $\{t_{i}(\overline{g})\mid i=1,\ldots,m\}$ to the cosets of $N$ in
$G$, where each $t_{i}$ is a word. There is a formula $\tau(\overline{g})$
(depending on $\kappa$) which asserts that $G=%
{\textstyle\bigcup\nolimits_{i=1}^{m}}
Nt_{i}(\overline{g})$.

Now we deal separately with cases (a) and (b).\medskip

\emph{Case (a)}: For each $i$ and $j$ we have $h_{i}^{g_{j}}=v_{ij}%
(\overline{h})$ for some word $v_{ij}$. Thus $G$ satisfies%
\[
\mathrm{conj}(\overline{g}):=\bigwedge_{i,j}\left[  g_{j}^{-1}w_{i}%
g_{j}=v_{ij}(w_{1},\ldots,w_{s})\right]  .
\]

For each $i$ there exist $i^{\ast}$ and a word $u_{i}$ such that $g_{i}%
=u_{i}(\overline{h})t_{i^{\ast}}(\overline{g})$. Then $G$ satisfies%
\begin{align*}
\rho(\overline{g})  &  :=\bigwedge_{i=1}^{r}\left[  g_{i}=u_{i}(w_{1}%
,\ldots,w_{s})t_{i^{\ast}}(\overline{g})\right] \\
\mathrm{extn}(\overline{g})  &  :=\bigwedge_{i,j}\left[  t_{i}(\overline
{g})t_{j}(\overline{g})=c_{ij}(w_{1},\ldots,w_{s})t_{s(i,j)}(\overline
{g})\right]
\end{align*}
for suitable words $c_{ij}$; here $(i,j)\longmapsto s(i,j)$ describes the
multiplication table of $G/N,$ and $(i,j)\longmapsto c_{ij}(\overline{h})$
represents the 2-cocycle defining the extension of $N$ by $G/N$; this takes
values in $\left\langle h_{1},\ldots,h_{s}\right\rangle $ because of
hypothesis (a).

Now suppose that $y_{1},\ldots,y_{r}\in H\in\mathcal{C}$ and that%
\begin{equation}
H\models\Phi(\overline{y})\wedge\tau(\overline{y})\wedge\rho(\overline
{y})\wedge\mathrm{conj}(\overline{y})\wedge\mathrm{extn}(\overline{y}).
\label{H-def}%
\end{equation}
Put $M=\kappa(H;\overline{y})$ and set $k_{i}=w_{i}(\overline{y})$ for
$i=1,\ldots,s$.

The fact that $H\models\Phi(\overline{y})$ implies that each $k_{i}\in M$ and
that the map sending $\overline{h}$ to $\overline{k}$ extends to an
isomorphism $\theta_{1}:N\rightarrow M$. Define $\theta:G\rightarrow H$ by%
\[
(at_{i}(\overline{g}))\theta=a\theta_{1}\cdot t_{i}(\overline{y})~\ \ (a\in
N,~1\leq i\leq m).
\]
Then $\tau(\overline{y})$ ensures that $\theta$ is a bijection, and using
$\mathrm{conj}(\overline{y})$ and $\mathrm{extn}(\overline{y})$ one verifies
that $\theta$ is a homomorphism; the key point is that $\mathrm{conj}%
(\overline{g})$ determines the conjugation action of each $g_{i}$ on $N$
because the $h_{j}$ generate $N$ topologically and inner automorphisms are
continuous, and similarly $\mathrm{conj}(\overline{y})$ determines the action
of each $y_{i}$ on $m$. This implies that for $b\in N$ and each $j$,%
\[
t_{j}(\overline{y})^{-1}\cdot b\theta_{1}\cdot t_{j}(\overline{y})=\left(
t_{j}(\overline{g})^{-1}\cdot b\cdot t_{j}(\overline{g})\right)  \theta_{1}.
\]

Finally, $\rho(\overline{y})$ implies that $g_{i}\theta=y_{i}$ for each $i$.

Thus (\ref{H-def}) implies that there is an isomorphism $G\rightarrow H$
sending $\overline{g}$ to $\overline{y}$.\medskip

\emph{Case (b)}: Assume now that $\mathrm{Z}(N)=1.$ Given a group $N$ with
trivial centre, a group $Q$, and a homomorphism $\gamma:Q\rightarrow
\mathrm{Out}(N)$, there is (up to equivalence) at most one extension group $G$
of $N$ by $Q$ such that conjugation in $G$ induces the mapping $\gamma
:Q\rightarrow\mathrm{Out}(N)$ (\cite{G}, \S 5.4, Theorem 2, Remark 1). So in
this case, it suffices to fix $N,~G/N$ and the action.

We fix the multiplication table of $G/N$ with%
\[
\mathrm{quot}(\overline{g}):=\bigwedge_{i,j}\kappa(\overline{g},t_{i}%
t_{j}t_{s(i,j)}{}^{-1})
\]
(writing $t_{i}$ in place of $t_{i}(\overline{g})$ throughout). We redefine
$\rho$ as follows:%
\[
\rho(\overline{g}):=\bigwedge_{i=1}^{r}\kappa\left(  \overline{g},t_{i^{\ast}%
}g_{i}^{-1}\right)  .
\]
where $i^{\ast}$ is defined above. To fix the action, we now set%
\[
\mathrm{conj}(\overline{g}):=\mathrm{res}\left(  \kappa(\overline
{g}),\bigwedge_{i,j}\phi_{v(i,j)}(w_{1},\ldots,w_{s},g_{j}^{-1}w_{i}%
g_{j})\right)
\]
where $v(i,j)=h_{i}^{g_{j}}$, and $\phi_{v(i,j)}$ defines $h_{i}^{g_{j}}$ in
$N$ in terms of $\overline{h}$.

Now suppose that $y_{1},\ldots,y_{r}\in H\in\mathcal{C}$ and that
\[
H\models\Phi(\overline{y})\wedge\tau(\overline{y})\wedge\rho(\overline
{y})\wedge\mathrm{quot}(\overline{y})\wedge\mathrm{conj}(\overline{y}).
\]
Put $M=\kappa(H;\overline{y})$ and set $k_{i}=w_{i}(\overline{y})$ for
$i=1,\ldots,s$. As before we have an isomorphism $\theta_{1}:N\rightarrow M$
sending $\overline{h}$ to $\overline{k}$. The map sending $t_{i}(\overline
{g})$ to $t_{i}(\overline{y})$ for each $i$ induces an isomorphism $\theta
_{2}:G/N\rightarrow H/M.$ Thus we have a diagram of group extensions:%
\[%
\begin{array}
[c]{ccccccccc}%
1 & \rightarrow & N & \rightarrow & G & \rightarrow & G/N & \rightarrow & 1\\
&  & \downarrow &  &  &  & \downarrow &  & \\
1 & \rightarrow & N & \overset{\alpha}{\rightarrow} & H & \overset{\beta
}{\rightarrow} & G/N & \rightarrow & 1
\end{array}
\]
where $\alpha:N\rightarrow M\hookrightarrow H$ and $\beta:H\twoheadrightarrow
H/M\rightarrow G/N$ are induced respectively by $\theta_{1}:N\rightarrow M$
and $\theta_{2}^{-1}:H/M\rightarrow G/N$, and the vertical arrows represent
identity maps. Now $\rho(\overline{g})$ and $\rho(\overline{y})$ ensure that
$(Ng_{i})\theta_{2}=My_{i}$ for each $i$. Then using $\mathrm{conj}%
(\overline{g})$ and $\mathrm{conj}(\overline{y})$ together with Remark (i), we
can verify that the two mappings $G/N\rightarrow\mathrm{Out}(N)$ induced by
the top extension and the bottom extension are identical. Hence there exists a
homomorphism $\theta:G\rightarrow H$ making the diagram commute, and then
$\theta$ must be an isomorphism since the end maps are bijective.\medskip

Finally, because $G$ is finitely generated, Theorem \ref{NSSC}\ (iii) ensures
that any group isomorphism $G\rightarrow H$ is a topological isomorphism.
\end{proof}

\bigskip

\textbf{Remark.} \ This argument gives the same result for a class
$\mathcal{C}$ of abstract groups, if we add the hypothesis that $N$ (has
finite index and) is definable in $G$.

\section{Bi-interpretation\label{bisecgr}}

Our definition of bi-interpretation will not be the most general one; instead
we introduce a few ad hoc definitions tailored to the purpose at hand. In
particular, the ring is supposed to be \emph{definable} as a \emph{closed}
subgroup of the group (rather than just `interpretable' as a subgroup of a
power of the group), and the group is supposed to be \emph{definable} as a
\emph{closed} subset of some power of the ring.

The purpose is to show that for certain groups built out of rings, finite
axiomatizability of the ring is equivalent to that of the group. Because we
are interested in profinite objects, some topological tweaks are added. The
precise result is Theorem \ref{biint} below; applications are given in
\S \ref{examplesubsec}.

\subsection{Interpreting rings in groups}

All rings are commutative, with identity. $L_{\mathrm{rg}}$~is the first-order
language of rings. A ring is \emph{profinite} if it is an inverse limit of
finite rings. We also need a slightly weaker version: the ring $R$ is
\emph{additively profinite} if its additive group $(R,+)$ is profinite as a group.

A profinite ring $R$ is \emph{FA} (resp. \emph{strongly FA}) if there is a
formula $\sigma$ of $L_{\mathrm{rg}}$ such that (i) $R\models\sigma$ and (ii)
if $S$ is a profinite (resp. additively profinite) ring and $S\models\sigma$
then $S$ is topologically isomorphic to $R$.

Let $R$ be an additively profinite ring. We say that $R$ is
\emph{topologically interpreted} in a profinite group $G$ if there are
formulae $\tau,\mu$ and a tuple of parameters $\overline{g}$ in $G$ with the
following property:

\begin{itemize}
\item for every profinite group $H$ and tuple $\overline{h}$ from $H$, the set
$\tau(\overline{h};H)$ is a closed subgroup of $H$;

\item $\tau(\overline{g};G)$ becomes a ring $\widehat{\tau}(G)=\widehat{\tau
}(\overline{g};G)$ with ring addition given by the group operation, and ring
multiplication defined by $\mathrm{res}(\mu(\overline{g}),\tau(\overline{g}%
))$, in the sense that for $x,y,z\in\tau(\overline{g};G),$%
\[
x\cdot y=z\Longleftrightarrow G\models\mu(\overline{g},x,y,z).
\]

\item $\widehat{\tau}(\overline{g};G)$ is topologically isomorphic to $R.$
\end{itemize}

(Here $\widehat{\tau}$ stands for $(\tau,\mu),$ and we will write
$\widehat{\tau}(G)$ for $\widehat{\tau}(\overline{g};G)$ when there is no risk
of confusion.)

\medskip

In this situation, there is a formula $\rho$ (depending on $\tau$ and $\mu$)
such that (i)~$G\models\rho(\overline{g})$ and (ii) for any profinite group
$H$ and tuple $\overline{h}$ from $H$, if $H\models\rho(\overline{h})$ then
the subgroup $\tau(\overline{h};H)$ is a ring $S:=\widehat{\tau}(\overline
{h};H)$ with operations defined as above. (The formula~$\rho$ expresses the
statements that $\mu$ defines a binary operation on $S$ and that the axioms
for a commutative ring with identity are satisfied). This ring $S$ will be
additively profinite, because $\tau(\overline{h};H)$ is a profinite group.

We call such an interpretation\emph{\ strongly topological} if it has the
following additional property: for any profinite group $H$ and tuple
$\overline{h}$ from $H$, if $H\models\rho(\overline{h})$ then the ring
$S=\widehat{\tau}(\overline{h};H)$ is actually a profinite ring: that is, the
multiplication map from $S\times S$ to $S$ is continuous.

For each formula $\phi$ of $L_{\mathrm{rg}}$ there is a formula $\phi^{\ast}$
of $L_{\mathrm{gp}}$ such that%
\[
\widehat{\tau}(\overline{h};H)\models\phi\Longleftrightarrow H\models
\phi^{\ast}(\overline{h}),
\]
obtained in the obvious way by translating each atomic $L_{\mathrm{rg}}$
subformula of $\phi$ into an equivalent $L_{\mathrm{gp}}$ formula.\medskip

\begin{lemma}
\label{psithing}Let $R$ be an FA profinite ring $R$. Suppose that $R$ is
topologically interpreted in a profinite group $G$, and assume further
\emph{either }that $R$ is strongly FA, \emph{or} that the interpretation is
strongly topological. Then there is an $L_{\mathrm{gp}}$ formula
$\psi(\overline{y})$ such that $G\models\psi(\overline{g})$ (where
$\overline{g}$ is as above), and for each profinite group $H$ and tuple
$\overline{h}$, if $H\models\psi(\overline{h})$ then $\widehat{\tau}%
(\overline{h};H)$ is a ring topologically isomorphic to $R$.
\end{lemma}

Indeed, it suffices to set $\psi(\overline{y})=\rho(\overline{y})\wedge
\sigma_{R}^{\ast}(\overline{y})$.

\medskip

\begin{remark}
\label{2-remark}The ring $R$ has the property `$2$ is not a zero divisor' if
and only if $\widehat{\tau}(G)$ satisfies a certain formula $\phi(\overline
{g})$. In this case, we make the convention that $\rho$ implies~$\phi^{\ast}$.
If $H$ as above now satisfies $\rho(\overline{h})$ then $2$ is not a zero
divisor in the ring $S=\widehat{\tau}(H)$, and then in $S$ the identity%
\begin{equation}
2xy=(x+y)^{2}-x^{2}-y^{2} \label{square-trick}%
\end{equation}
determines $xy$. Since addition is continuous, by definition of the topology
on $S$, to establish continuity of multiplication it will suffice to show that
the map $x\longmapsto x^{2}$ on $S$ is continuous.
\end{remark}

Thus if $2$ is not a zero divisor in $R$, for the interpretation be
topological it suffices to have: whenever $H$ as above satisfies
$\rho(\overline{h})$, the squaring map from $S=\widehat{\tau}(H)$ to $S$ is continuous.

\subsection{Interpreting groups in rings\label{bisec}}

Let $G$ be a profinite group. We say that $G$ is \emph{topologically}
\emph{interpreted} in a profinite ring $R$ if, for some $d$, there are
$L_{\mathrm{rg}}$ formulae $\alpha_{1},~\alpha_{2}$ such that

\begin{itemize}
\item for every profinite ring $T$, the subset $\alpha_{1}(T^{(d)})$ is closed
in $T^{(d)}$;

\item $\alpha_{1}(R^{(d)})$ is a group $\widehat{\alpha}(R)$, with operation
defined by%
\[
\overline{a}\cdot\overline{b}=\overline{c}\Longleftrightarrow R\models
\alpha_{2}(\overline{a},\overline{b},\overline{c});
\]

\item $G$ is topologically isomorphic to $\widehat{\alpha}(R)$ (with the
subspace topology induced by $\overline{\alpha}(R)\subseteq R^{(d)}$).
\end{itemize}

As in the preceding subsection, there is a formula $\alpha_{3}$ (depending on
$\alpha_{1},~\alpha_{2}$) such that (i) $R\models\alpha_{3}$ and (ii) for any
profinite ring $T$, if $T\models\alpha_{3}$ then $\alpha_{1}(T^{(d)})$ is a
group $\widehat{\alpha}(T)$ with the operation defined as above.

\medskip

The interpretation is \emph{strongly topological} if, in addition, for every
profinite ring $T,$ the group operation defined by $\alpha_{2}$ on
$\widehat{\alpha}(T)$ is continuous; in this case, $\widehat{\alpha}(T)$ will
be a profinite group, as the topology it inherits from $T^{(d)}$ is compact
and totally disconnected.

For example, if $\mathfrak{G}\leq\mathrm{SL}_{n}$ is an algebraic group
defined over $\mathbb{Z}$, then $\mathfrak{G}(R)$ is interpreted in $R$ for
any ring $R;$ here $d=n^{2}$, $\alpha_{1}$ expresses the defining equations of
$\mathfrak{G}$, and $\alpha_{2}$ is the formula for matrix multiplication
(which is continuous when $R$ is a topological ring).

Now let $G$ be a profinite group and $R$ a profinite ring. We say that $G$ and
$R$ are \emph{topologically bi-interpretable }in the following circumstances:

\begin{enumerate}
\item $R$ is topologically interpreted in $G$ by $\tau$ and $G$ is
topologically interpreted in $R$ by $\overline{\alpha}$, as above;

\item identifying $G$ with $\widehat{\alpha}(R)\subseteq R^{(d)}$ and $R$ with
$\widehat{\tau}(G)\subseteq G$ gives two mappings%
\begin{align*}
\theta:G  &  =\widehat{\alpha}(R)\rightarrow\widehat{\alpha}(\widehat{\tau
}(G))\hookrightarrow G^{(d)},\\
\theta^{\prime}:R  &  =\widehat{\tau}(G)\rightarrow\widehat{\tau
}(\widehat{\alpha}(R))\hookrightarrow R^{(d)};
\end{align*}
then $\theta$ is $L_{\mathrm{gp}}$-definable and $\theta^{\prime}$ is
$L_{\mathrm{rg}}$-definable.
\end{enumerate}

Note that in this situation, $\theta$ is a group isomorphism from $G$ to
$\widehat{\alpha}(\widehat{\tau}(G)),$ and $\theta^{\prime}$ is a ring
isomorphism from $R$ to $\widehat{\tau}(\widehat{\alpha}(R))$.

While first-order language may suffice to determine the algebraic structure of
a group, it cannot say anything about the topology. Recall that the profinite
group $G$ is \emph{algebraically rigid} if every profinite group abstractly
isomorphic to $G$ is topologically isomorphic to $G$. This holds in particular
if $G$ is strongly complete (i.e. every subgroup of finite index is open), but
the conditions are not equivalent; in \S \ref{examplesubsec} we will exhibit
groups that are FA, and therefore algebraically rigid, but not strongly complete.

Analogously, a profinite ring is said to be algebraically rigid if its
topology is uniquely determined by its algebraic structure.

\begin{theorem}
\label{biint} Let $G$ be a profinite group, $R$ a profinite ring, and suppose
that $G$ and $R$ are topologically bi-interpretable.

\emph{(i)} Suppose that $R$ is FA. Assume that $G$ is algebraically rigid, and
that the interpretation of $R$ in $G$ is strongly topological or that $R$ is
strongly FA. Then $G$ is FA in profinite groups.

\emph{(ii) }Suppose that $G$ is FA in profinite groups. Assume that $R$ is
algebraically rigid, and that the interpretation of $G$ in $R$ is strongly
topological. Then $R$ is FA.
\end{theorem}

\begin{proof}
(i) Let $\psi(\overline{y})$ be the formula provided by Lemma \ref{psithing}.
Given that $\theta$ is definable, the statement that $\theta$ is an
isomorphism from $G$ onto $\widehat{\alpha}(\widehat{\tau}(G))$ can be
expressed by a certain $L_{\mathrm{gp}}$ formula $\Theta(\overline{g})$,
depending in a straightforward way on $\overline{\alpha}$ and $\overline{\tau
}$. Then $G\models\Sigma_{G}(\overline{g})$ where
\[
\Sigma_{G}(\overline{y})\equiv\psi(\overline{y})\wedge\Theta(\overline{y}).
\]

Now let $H$ be a profinite group and suppose that $H\models\Sigma
_{G}(\overline{h})$ for some tuple $\overline{h}$ in $H$. As $H\models
\psi(\overline{h})$, the ring $S=\widehat{\tau}(\overline{h};H)$ is
topologically isomorphic to $R$. In particular, $S$ is a profinite ring and
$S\models\alpha_{3}$, so $\widehat{\alpha}(S)$ is a group with operation
defined by $\alpha_{2}$. As $H\models\Theta(\overline{h})$, the formula
defining $\theta$ establishes a group isomorphism
\[
\widetilde{\theta}:H\rightarrow\widehat{\alpha}(S)\cong\widehat{\alpha
}(R)\cong\widehat{\alpha}(\widehat{\tau}(\overline{g};G)).
\]
Then $\theta\widetilde{\theta}^{-1}$ is a group isomorphism $G\rightarrow H$.
As $H$ is a profinite group and $G$ is algebraically rigid, the groups are
topologically isomorphic.

Thus $\exists\overline{y}.\Sigma_{G}(\overline{y})$ determines $G$ as a
profinite group.

(ii) This is similar, swapping the roles of $G$ and $R$.
\end{proof}

\bigskip

\subsection{Some profinite rings\label{profringsec}}

Familiar examples of profinite rings are the complete local rings with finite
residue field: if $R$ is one of these, with (finitely generated) maximal ideal
$\mathfrak{m}$ and finite residue field $R/\mathfrak{m}\cong\mathbb{F}_{q}$,
then $R$ is the inverse limit of the finite rings $R/\mathfrak{m}^{n}$
($n\in\mathbb{N}$). We will keep this notation throughout this section, and
set $p=\mathrm{char}(R/\mathfrak{m})$, $q=p^{f}$.

The fundamental structure theorem of I. S. Cohen describes most of these rings
quite explicitly. $R$ is said to be \emph{regular} if $\mathfrak{m}$ can be
generated by $d$ elements where $d=\dim R$ is the Krull dimension of $R$. Also
$R$ is said to be \emph{unramified} if either $pR=0$ or $p\cdot1_{R}%
\notin\mathfrak{m}^{2}$. (For background on regular local rings, see e.g.
\cite{E}, \S 10.3. The `unramified' condition serves to avoid complications in
the case of unequal characteristic.)

A basic example is the complete discrete valuation ring%
\[
\mathfrak{o}_{q}=\mathbb{Z}_{p}[\zeta_{q-1}]
\]
where $\zeta_{q-1}$ is a primitive $(q-1)$th root of unity; this is the ring
of integers in the unique unramified extension of degree $f$ over
$\mathbb{Q}_{p}$. According to \cite{ISC}, Th. 11, Cor.~2 this is the only
complete local domain $R$ of characteristic $0$ with maximal ideal $pR$ and
residue field $\mathbb{F}_{q}$.

\begin{theorem}
\emph{\label{cohenthm}(\cite{ISC}, Theorem 15) \ }Let $R$ be a regular,
unramified complete local ring with residue field $\mathbb{F}_{q}$ and
dimension $d\geq1$. Then one of the following holds:

\emph{(a)} $R\cong\mathbb{F}_{q}[[t_{1},\ldots,t_{d}]]$

\emph{(b)} $R\cong\mathfrak{o}_{q}[[t_{1},\ldots,t_{d-1}]]$.
\end{theorem}

The point is that $R$ is determined up to isomorphism by its characteristic
and the parameters $d,~q$; it is then hardly surprising that such a ring is
FA. (The same very likely holds also in the ramified case, when $R$ is an
`Eisenstein extension' of the ring specified in (b); we shall not go into this
here, but it can probably be covered by suitably extending the arguments below.)

\begin{theorem}
\label{regloc}Let $R$ be a regular, unramified complete local ring with finite
residue field$.$ Then $R$ is FA.
\end{theorem}

\begin{proof}
Until further notice $S$ denotes an arbitrary ring. Each of the statements
`$S$ is an integral domain', `$\mathrm{char}S=0$', `$\mathrm{char}S=p$' is
easily expressible by a sentence of $L_{\mathrm{rg}}$. There are formulae
$\mu,~\varphi_{q}$ and $\rho$ such that

\begin{itemize}
\item $S\models\mu(a_{1},\ldots,a_{d})$ if and only if $S\setminus\sum_{1}%
^{d}a_{i}S$ consists of units;

\item $S\models\varphi_{q}(a_{1},\ldots,a_{d})$ if and only if $\left\vert
S/\sum_{1}^{d}a_{i}S\right\vert =q$;

\item $S\models\rho(a_{1},\ldots,a_{d})$ if and only if for each $i$, the
element $a_{i}$ is not a zero divisor modulo $\sum_{j=1}^{i-1}a_{j}S$ (the
zero ideal when $i=1$).
\end{itemize}

Put%
\[
\Sigma_{q}(\overline{x}):=\mu(\overline{x})\wedge\varphi_{q}(\overline
{x})\wedge\rho(\overline{x})\wedge\forall y,z.(yz=0\longrightarrow(y=0\vee
z=0)).
\]
Now suppose that $S$ satisfies $\Sigma_{q}(a_{1},\ldots,a_{d})$, and set
$I=\sum_{1}^{d}a_{i}S.$ Then $S$ is a local domain with maximal ideal $I$ and
residue field $S/I\cong\mathbb{F}_{q}$. The sentence $\rho(a_{1},\ldots
,a_{d})$ implies that $\dim S$ is at least $d$, and hence that $\dim S=d$
(\cite{ISC}, Theorem 14), so $S$ is regular.

Now we separate cases.\smallskip

\emph{Case (a)}: $R\cong\mathbb{F}_{q}[[t_{1},\ldots,t_{d}]]$. Then $R$
satisfies
\[
\Sigma_{q,p}(t_{1},\ldots,t_{d})\equiv\Sigma_{q}(t_{1},\ldots,t_{d}%
)\wedge\forall y.(py=0).
\]
Suppose that $S$ is a profinite ring and that $S\models\Sigma_{q,p}%
(s_{1},\ldots,s_{d})$ for some $s_{1},\ldots,s_{d}\in S$. Put $I=\sum_{1}%
^{d}s_{i}S$. Then $S$ is a regular, unramified local domain of dimension $d$
and characteristic $p$, with maximal ideal $I$ and residue field
$\mathbb{F}_{q}$.\smallskip

\emph{Case (b)}: $R\cong\mathfrak{o}_{q}[[t_{1},\ldots,t_{d-1}]].$ Note that
$1_{R}$ is a definable element, by the formula $\forall y.(xy=y)$. Now
$\mathfrak{m}=pR+\sum_{1}^{d-1}t_{i}R$. The fact that $p1_{R}\notin%
\mathfrak{m}^{2}$ is expressible by%
\[
\tau_{p}(t_{1},\ldots,t_{d-1}):=\forall z_{ij},y_{i},x.\left(  \sum_{i,j}%
t_{i}t_{j}z_{ij}+\sum_{i}pt_{i}y_{i}+p^{2}x\neq p1_{R}\right)  .
\]
Thus $R$ satisfies%
\[
\Sigma_{q,0}(t_{1},\ldots,t_{d-1})\equiv\Sigma_{q}(p1_{R},t_{1},\ldots
,t_{d-1})\wedge\exists y.(py\neq0)\wedge\tau_{p}(t_{1},\ldots,t_{d-1}).
\]
Suppose now that $S$ is profinite ring and that $S\models\Sigma_{q,0}%
(s_{1},\ldots,s_{d-1})$ for some $s_{1},\ldots,s_{d-1}\in S$. Put
$I=pS+\sum_{1}^{d-1}s_{i}S$. Then again, $S$ is a regular, unramified local
domain of dimension $d$, with maximal ideal $I$ and residue field
$\mathbb{F}_{q}$, and $S$ has characteristic $0$.\smallskip

\emph{Conclusion. } Since ring multiplication is continuous, $I$ is compact
and therefore closed in $S$; as it has finite index, $I$ is open. The same
argument shows that $I^{n}$ is open for each $n$ (each $I^{n-1}/I^{n}$ is
finite because it is finitely generated as a module for $S/I$). Now $S$ is the
inverse limit of finite rings $S/J_{\alpha}$, where $\{J_{a}\}$ is a family of
open ideals that form a base for the neighbourhoods of $0$. For each $\alpha$
the ring $S/J_{\alpha}$ is finite with Jacobson radical $I/J_{\alpha}$, so for
some $n$ we have $I^{n}\subseteq J_{\alpha}$. Hence the system $\{I^{n}\mid
n\in\mathbb{N\}}$ is also a base for the neighbourhoods of $0$. Thus the given
profinite topology is the $I$-adic topology, and as $S$ is complete for the
former it is complete as a local ring.

Now Theorem \ref{regloc} shows that $S\cong\mathbb{F}_{q}[[t_{1},\ldots
,t_{d}]]$ (Case a) or $S\cong\mathfrak{o}_{q}[[t_{1},\ldots,t_{d-1}]]$ (Case b).
\end{proof}

\medskip

\textbf{Remarks.} (i) If $R$ is one-dimensional and of characteristic $0,$
i.e. $R\cong\mathfrak{o}_{q}$ for some $q$, then in fact $R$ is \emph{strongly
FA. }Indeed, in Case (b) above we merely need to assume that $S$ is additively
profinite, for $S$ satisfies $\Sigma_{q,0}$ with $d=1,$ which asserts that
$I=pS.$ As multiplication by $p$ is continuous on the profinite additive group
$(S,+)$, the given argument shows that the powers of $I$ define the topology,
which again implies that $S$ is complete.

(ii) We can deduce: \emph{If} $\pi$ \emph{is a finite set of primes then}
$\mathbb{Z}_{\pi}$ \emph{is strongly FA}. Taking a set of pairwise orthogonal
idempotents $e_{1},\ldots,e_{k}$ as parameters, we find that $R$ is determined
by a first-order statement which asserts:

\begin{itemize}
\item $R=e_{1}R\times\ldots\times e_{k}R$ and $e_{i}R\cong\mathbb{Z}_{p_{i}}$
for $i=1,\ldots,k$.
\end{itemize}

(iii) We used the well-known fact that in a finite (more generally, Artinian)
ring, the Jacobson radical is nilpotent. It is worth stating an immediate consequence.

\begin{lemma}
\label{niljac} Let $R$ be a complete local ring with finite residue field.
Then every ideal of finite index in $R$ is open.
\end{lemma}

\noindent In other words, rings of this kind are `strongly complete'. In
particular, they are algebraically rigid.

\subsection{Some worked examples\label{examplesubsec}}

Now we are ready to prove Theorem \ref{sl2rt}: \medskip

\noindent\emph{if} $R$ \emph{is a complete, unramified regular local ring with
finite residue field then each of the profinite groups} $\mathrm{Af}_{1}(R)$,
$\mathrm{SL}_{2}(R)$ \emph{is FA in the class of profinite groups, assuming in
the second case that the residue characteristic is odd.} \medskip

This will follow from Theorem \ref{regloc} once we show that the hypotheses of
Theorem \ref{biint}(i) are satisfied.

Let $R$ be a complete local domain, with maximal ideal $\mathfrak{m}$ and
finite residue field $R/\mathfrak{m}$. Then $R$ is a profinite ring, a base
for the neighbourhoods of $0$ being the powers of $\mathfrak{m}$. In
particular, $R^{\ast}=R\setminus\mathfrak{m}$ is an open subset, being the
union of finitely many additive cosets of $\mathfrak{m}$.

The semi-direct product $(R,+)\rtimes R^{\ast}$ can be identified with the
affine group%
\[
\mathrm{Af}_{1}(R)=\left(
\begin{array}
[c]{cc}%
1 & R\\
0 & R^{\ast}%
\end{array}
\right)  <\mathrm{GL}_{2}(R).
\]

\begin{theorem}
\label{affine}If $R$ is FA then the group $\mathrm{Af}_{1}(R)$ is FA in
profinite groups.
\end{theorem}

\begin{proof}
We will verify the hypotheses of Theorem \ref{biint}(i). Write $G=\mathrm{Af}%
_{1}(R).$ We have a visible topological interpretation of $G$ as a subset of
$M_{2}(R)=R^{(4)}$.

Define the following elements of $G,$ where $1=1_{R}$ and $\lambda\in
R,~\xi\in R^{\ast}:$%
\begin{equation}
u(\lambda)=\left(
\begin{array}
[c]{cc}%
1 & \lambda\\
0 & 1
\end{array}
\right)  ,~~h(\xi)=\left(
\begin{array}
[c]{cc}%
1 & 0\\
0 & \xi
\end{array}
\right)  ,\nonumber
\end{equation}
and fix the parameters $u:=u(1),$ $h:=h(r)$ for some $r\in R^{\ast},$ $r\neq
1$. Then $G=H\cdot U$ where
\begin{align*}
U  &  :=u(R)=\mathrm{C}_{G}(u)\\
H  &  :=h(R^{\ast})=\mathrm{C}_{G}(h)
\end{align*}
are both definable subgroups. For technical reasons, we want to encode the
fact that $H$ is abelian and normalizes $U$; to this end, we define%
\[
\kappa(u,h,x)\equiv\forall y.\left(  [\mathrm{com}(y,h)\rightarrow
\mathrm{com}(y,x)]\wedge\lbrack\mathrm{com}(y,u)\longleftrightarrow
\mathrm{com}(y^{x},u)]\right)  ;
\]
and note that $H=\kappa(u,h;G)$.

We will frequently use the identity%
\[
u(\lambda)^{h(\xi)}=u(\xi\lambda).
\]
All formulae are supposed to involve the parameters $u,~h$, which we will
sometimes omit for brevity.\medskip

\noindent\textbf{Claim 1.} \emph{The ring} $R$ \emph{has a strongly
topological interpretation in }$G$\emph{, given by} $u:R\rightarrow
U=\mathrm{com}(u;G)$.\medskip

Certainly $\mathrm{com}(u,-)$ defines a closed subgroup in any profinite
group, as it defines a centralizer. The map $u:R\rightarrow U$ is a
topological isomorphism from $(R,+)$ to $U$. It becomes a topological ring
isomorphism if one defines%
\[
u(\alpha)\cdot u(\beta)=u(\alpha\beta).
\]
We need to provide an $L_{\mathrm{gp}}$ formula $\mu$ such that for
$x,~y,~z\in U$,%
\[
x\cdot y=z\Longleftrightarrow G\models\mu(u,h;x,y,z).
\]
Let $v=u(\beta)\in U.$ If $\beta\in R^{\ast}$ then $v=u^{h(\beta)}$, while if
$\beta\in\mathfrak{m}$ then $\beta+1\in R^{\ast}$ and $v=[u,h(\beta+1)]$. Thus
$v_{1}\cdot v_{2}=v_{3}$ if and only if there exist $x,y\in H$ such that one
of the following holds:%
\begin{align*}
v_{1}  &  =u^{x},~v_{2}=u^{y},~v_{3}=u^{xy},~\text{or}\\
v_{1}  &  =u^{x},~v_{2}=[u,y],~v_{3}=[u^{x},y],\text{ or}\\
v_{1}  &  =[u,x],~v_{2}=u^{y},~v_{3}=[u^{y},x],\text{ or}\\
v_{1}  &  =[u,x],~v_{2}=[u,y],~v_{3}=[[u,y],x].
\end{align*}
This can be expressed by a first-order formula since $H$ is definable.

To say that the interpretation is strongly topological means the following: if
a profinite group $\widetilde{G}$ satisfies the appropriate sentence $\rho$
with parameters $\widetilde{u},~\widetilde{h}$, which in particular implies
that $\mu(\widetilde{u},\widetilde{h};~-~)$ defines a binary operation $\cdot$
on $\widetilde{U}=\mathrm{com}(\widetilde{u};\widetilde{G}) $, this operation
is continuous.

We will write $u$ in place of $\widetilde{u}$ for aesthetic reasons.

Let $N$ be an open normal subgroup of $\widetilde{G}$. If $u^{x}\equiv
u^{x^{\prime}}~(\operatorname{mod}N)$ and $u^{y}\equiv u^{y^{\prime}%
}~(\operatorname{mod}N)$ with $x,x^{\prime},y,y^{\prime}\in\widetilde{H},$
then%
\[
u^{xy}\equiv u^{x^{\prime}y}=u^{yx^{\prime}}\equiv u^{y^{\prime}x^{\prime}%
}=u^{x^{\prime}y^{\prime}}~(\operatorname{mod}N)
\]
since $\widetilde{H}$ is abelian. Similar congruences hold if $u^{x}$ is
replaced by $[u,x]$ or $u^{y}$ is replaced by $[u,y].$ Thus in all cases we
see that if $v_{i}\equiv v_{i}^{\prime}~(\operatorname{mod}N)~$\ for $i=1,2$
and $v_{1}\cdot v_{2}=v_{3}$, $v_{1}^{\prime}\cdot v_{2}^{\prime}%
=v_{3}^{\prime},$ then $v_{3}\equiv v_{3}^{\prime}~(\operatorname{mod}N)$.
Thus the operation $\cdot$ is continuous as required.

\medskip

\noindent\textbf{Claim 2. }\emph{The map} $\theta$ \emph{sending }$g=\left(
\begin{array}
[c]{cc}%
1 & \lambda\\
0 & \xi
\end{array}
\right)  \in G$ to $\left(  u(1),u(\lambda),u(0),u(\xi)\right)  \in G^{(4)}$
is definable.

We have $g=\widetilde{h}(g)\widetilde{u}(g)$ where $\widetilde{u}%
(g)=u(\lambda)$ and $\widetilde{h}(g)=h(\xi).$ Also%
\begin{align*}
\{\widetilde{u}(g)\}  &  =Hg\cap U\\
\{\widetilde{h}(g)\}  &  =gU\cap H.
\end{align*}
As $H$ and $U$ are definable subsets of $G$, the mappings $\widetilde{u}%
:G\rightarrow G$ and $\widetilde{h}:G\rightarrow G$ are both definable. Hence
so is $\theta$, because $u(1)$ is the parameter $u$, $u(0)=1_{G},$
$u(\lambda)=\widetilde{u}(g)$ and%
\[
u(\xi)=u^{h(\xi)}=u^{\widetilde{h}(g)}.
\]

\medskip\noindent\textbf{Claim 2}bis.\emph{The map sending} $\lambda\in R$ to
$(1,\lambda,0,1)\in R^{(4)}$ \emph{is definable. }Obviously.

\medskip

\noindent\textbf{Claim 3. }$G$ \emph{is algebraically rigid}. This follows
from the stronger result Proposition \ref{aff}, below, and completes the proof
of Theorem \ref{affine}.
\end{proof}

\begin{proposition}
\label{aff} Let $R$ be a complete local domain with finite residue field. Then
every group isomorphism from $\mathrm{Af}_{1}(R)$ to a profinite group is
continuous (and therefore a topological isomorphism).
\end{proposition}

\begin{proof}
$G=\mathrm{Af}_{1}(R)=U\rtimes H$ where $U=u(R)$ and $H=h(R^{\ast})$ (notation
as above). A base for the neigbourhoods of $1$ in $G$ is the family of
subgroups%
\begin{align*}
G(n)  &  :=H(n)U(n),~~n\geq1,~\text{where}\\
U(n)  &  =u(\mathfrak{m}^{n}),~~H(n)=h(1+\mathfrak{m}^{n}).
\end{align*}

Let $\theta:G\rightarrow\widetilde{G}$ be a group isomorphism, where
$\widetilde{G}$ is a profinite group. Set $\widetilde{U}=U\theta$ and
$\widetilde{H}=H\theta$. As $R$ is an integral domain, $H=\mathrm{C}_{G}(H),$
and so $\widetilde{H}=\mathrm{C}_{\widetilde{G}}(\widetilde{H})$ is closed in
$\widetilde{G}$. Similarly, $\widetilde{U}=\mathrm{C}_{\widetilde{G}%
}(\widetilde{U})$ is closed. We will show that $\theta^{-1}$ is continuous
(which, for an isomorphism of profinite groups, is equivalent to being a homeomorphism).

Suppose that $\widetilde{N}\vartriangleleft_{o}\widetilde{G}$. Then
$N:=\widetilde{N}\theta^{-1}$ is a normal subgroup of finite index in $G$, so
$N\cap U=u(B)$ for some additive subgroup $B$ of finite index in $R$. If $r\in
R^{\ast}$ then
\[
u(Br)=u(B)^{h(r)}=u(B)
\]
so $B=Br$, and as $R=R^{\ast}\cup(R^{\ast}-1)$ it follows that $B$ is an ideal
of $R$; therefore $B\supseteq\mathfrak{m}^{n}$ for some $n$, by Lemma
\ref{niljac}. Thus $U(n)\subseteq\widetilde{N}\theta^{-1}$.

It follows that $\theta_{\left\vert U\right.  }:U\rightarrow\widetilde{U}$ is
a continuous isomorphism, and consequently a homeomorphism. This in turn
implies that each $\widetilde{U}(n):=U(n)\theta$ is open in $\widetilde{U}$.

Now for each $n$,%
\[
H(n)\theta=\mathrm{C}_{H}(U/U(n))\theta=\mathrm{C}_{\widetilde{H}%
}(\widetilde{U}/\widetilde{U}(n))
\]
is open in $\widetilde{H}$ since $\widetilde{U}(n)$ is open in $\widetilde{U}
$ (here $\mathrm{C}_{H}(U/U(n))$ denotes the kernel of the conjugation action
of $H$ on the factor $U/U(n)$ ). Thus%
\[
G(n)\theta=H(n)\theta.U(n)\theta
\]
is closed in $\widetilde{G},$ hence open as it has finite index. It follows
that $\theta^{-1}$ is continuous.
\end{proof}

\medskip

\textbf{Remark. } It is not usually the case that $G=\mathrm{Af}_{1}(R)$ is
strongly complete. In fact, $G$ is strongly complete if and only if its open
subgroup $G(1)$ (the principal congruence subgroup $\mathrm{mod}$
$\mathfrak{m}$) is, and as $G(1)$ is a pro-$p$ group this holds if and only if
$G(1)$ is finitely generated (e.g. by \cite{WS}, Theorem 2: a pro-$p$ group
that is not f.g.\ maps onto an infinite elementary abelian $p$-group, and so
has uncountably many subgroups of index $p$). This in turn holds if and only
if the multiplicative subgroup $T:=1+\mathfrak{m}$ of $R^{\ast}$ is finitely
generated as a pro-$p$ group. Now $T$ is finitely generated iff $T/T^{p}$ is
finite; of the rings listed in Theorem \ref{cohenthm}, only the ones denoted
$\mathfrak{o}_{q}$ have this property.

\bigskip

For the next example, let $R$ be as above, and \textbf{assume that}
$q=\left\vert R/\mathfrak{m}\right\vert $ \textbf{is odd}. Every element of
$R/\mathfrak{m}$ is then a difference of two squares, and as $R$ is complete
and $q$ is odd it follows that every element of $R$ is of the form $x^{2}$ or
$x^{2}-y^{2}$ with $x,$ $y\in R^{\ast}$. Fix $r\in R$ such that
$r+\mathfrak{m}$ is a generator for $(R/\mathfrak{m)}^{\ast}$ and $r^{4}\neq1$.

\begin{theorem}
\label{sl2_biint copy(1)}If $R$ is FA then the group $\mathrm{SL}_{2}(R)$ is
FA in profinite groups.
\end{theorem}

\begin{proof}
Put $G=\mathrm{SL}_{2}(R).$ We will show that all hypotheses of Theorem
\ref{biint} are satisfied, with a strongly topological interpretation of $R$
in $G$.

Define the following elements of $G,$ where $1=1_{R}$ and $\lambda\in R:$%
\begin{align}
u(\lambda)  &  =\left(
\begin{array}
[c]{cc}%
1 & \lambda\\
0 & 1
\end{array}
\right)  ,~~v(\lambda)=\left(
\begin{array}
[c]{cc}%
1 & 0\\
-\lambda & 1
\end{array}
\right)  ,\nonumber\\
h(\lambda)  &  =\left(
\begin{array}
[c]{cc}%
\lambda^{-1} & 0\\
0 & \lambda
\end{array}
\right)  ~\ (\lambda\in R^{\ast}),\label{names}\\
w  &  =\left(
\begin{array}
[c]{cc}%
0 & 1\\
-1 & 0
\end{array}
\right) \nonumber
\end{align}
We take $u:=u(1),$ $u^{\prime}:=u(r),v:=v(1)$, $w$ and $h:=h(r)$ as parameters
and write%
\[
\overline{u}=(u,u^{\prime},v,w,h).
\]
`Definable' will mean definable with these parameters. All formulae below are
supposed to include these parameters, which we mostly omit for brevity. We
will use without special mention the identities%
\begin{align*}
u(\lambda)^{w}  &  =v(\lambda)\\
u(\lambda)^{h(\mu)}  &  =u(\lambda\mu^{2}).
\end{align*}

Write $U=u(R),~V=v(R),~H=h(R^{\ast})$. Write $\pm U=\left\langle
\pm1\right\rangle \cdot U$, etc. Then $\pm U=\mathrm{C}_{G}(u),$ $\pm
V=\mathrm{C}_{G}(v)$ are definable subgroups of $G$.

To show that $U$ is definable, observe first that no element of $-U$ is
conjugate to an element of $U$. On the other hand, we shall see below that
each element of $U$ takes one of the forms $u^{x}$ or $u^{x}u^{-y}$ ($x,y\in
H$). Thus $U=\rho(G)$ where%
\[
\rho(s):=\mathrm{com}(s,u)\wedge\exists x,y.(s=u^{x}\vee s=u^{x}u^{-y}).
\]

Note that $\rho$ will define a closed subset in any profinite group, since
centralizers and conjugacy classes are closed. Using Lemma \ref{subgp} below,
we now adjust $\rho$ to a new formula $\rho^{\ast},$ such that $U=\rho^{\ast
}(G),$ and for any profinite group $A$ the subset $\rho^{\ast}(A)$ is a closed subgroup.

Then $V=U^{w}$ is also definable, as is the subgroup $H=\mathrm{C}_{G}(h)$;
for technical reasons, we want to encode the fact that $H$ is abelian and
normalizes $U$; thus $H=$ $\eta(G)$ where
\[
\eta(\overline{u},x)\equiv\forall y.\left(  \mathrm{com}(h,y)\rightarrow
\mathrm{com}(x,y)\wedge\rho^{\ast}(y)\longleftrightarrow\rho^{\ast}%
(y^{x})\right)  .
\]

\noindent\textbf{Claim 1.} \emph{The map} $u:R\rightarrow U=\rho^{\ast}%
(G)$\emph{ gives a strongly topological interpretation of }$R$ \emph{in }%
$G$.\medskip

We adapt the method used in \cite{KRT}, proof of Theorem 3.2. The map
$u:R\rightarrow U$ is a topological isomorphism from $(R,+)$ to $U$. It
becomes a ring isomorphism if one defines%
\begin{equation}
u(\alpha)\cdot u(\beta)=u(\alpha\beta). \label{thing}%
\end{equation}
Now we need to provide an $L_{\mathrm{gp}}$ formula $\mu$ such that for
$v_{1},~v_{2},~v_{3}\in U$,%
\[
v_{1}\cdot v_{2}=v_{3}\Longleftrightarrow G\models\mu(v_{1},v_{2},v_{3}).
\]

If $\alpha=\xi^{2}$ (resp. $\xi^{2}-\eta^{2}$) with $\xi,\eta\in R^{\ast},$
then $u(\alpha)=u^{x}$ (resp. $u^{x}u^{-y})$ where $x=h(\xi),$ $y=h(\eta)$,
and%
\[
u(\alpha^{2})=u^{x^{2}},\text{ resp. }u^{x^{2}}u^{-2xy}u^{y^{2}}.
\]

Let $\mathrm{sq}(v_{1},v_{2})$ be the formula asserting that there exist $x,y
$ in $H$ such that
\[
\left[  v_{1}=u^{x}\wedge v_{2}=u^{x^{2}}\right]  \vee\left[  v_{1}%
=~u^{-y}u^{x}\wedge v_{2}=u^{x^{2}}u^{-2xy}u^{y^{2}}\right]  ~.
\]
One verifies easily that this holds if and only if $v_{2}=v_{1}\cdot v_{1}$ in
the sense of (\ref{thing}). Now in view of (\ref{square-trick}) we can take
$\mu(v_{1},v_{2},v_{3})$ to assert that there exist $a$ and $b$ such that%
\[
\mathrm{sq}(v_{1},a)\wedge\mathrm{sq}(v_{2},b)\wedge\mathrm{sq}(v_{1}%
v_{2},av_{3}^{2}b).
\]

To complete the proof of {Claim 1}, it remains to establish that the
interpretation is strongly topological. In view of Remark \ref{2-remark}, it
will suffice to show that $\mathrm{sq}(v_{1},v_{2})$ defines a continuous map
\ -- \ not just on $U$ but on any profinite group arising as $\rho^{\ast
}(\overline{u}^{\ast};G^{\ast})$ where $G^{\ast}$ satisfies the appropriate
sentences (which include one asserting that $\mathrm{sq}(v_{1},v_{2})$ does
define a mapping). For simplicity (`by abuse of notation') we keep the
notation attached to $G$, but will not use any special properties of $G$. The
fact that $H=\eta(\overline{u};G)$ is abelian and normalizes $U$ is now
implied by the definition of $\eta.$

If $N$ is an open normal subgroup of $G,$ $u\in U$ and $x,y,x^{\prime
},y^{\prime}\in H$ then $u^{x}\equiv u^{x^{\prime}}$ $(\operatorname{mod}N)$
implies
\[
u^{x^{2}}\equiv u^{x^{\prime}x}=u^{xx^{\prime}}\equiv u^{x^{\prime2}%
}~(\operatorname{mod}N).
\]
\ Similarly, $u^{-y}u^{x}\equiv u^{-y^{\prime}}u^{x^{\prime}}%
~(\operatorname{mod}N)$ implies%
\begin{align*}
u^{x^{2}}u^{-2xy}u^{y^{2}}  &  =(u^{-y}u^{x})^{x}(u^{-x}u^{y})^{y}\\
&  \equiv(u^{-y^{\prime}}u^{x^{\prime}})^{x}(u^{-x^{\prime}}u^{y^{\prime}%
})^{y}=(u^{x}u^{-y})^{x^{\prime}}(u^{-x}u^{y})^{y^{\prime}}\\
&  \equiv(u^{x^{\prime}}u^{-y^{\prime}})^{x^{\prime}}(u^{-x^{\prime}%
}u^{y^{\prime}})^{y^{\prime}}=u^{x^{\prime2}}u^{-2x^{\prime}y^{\prime}%
}u^{y^{\prime2}}.
\end{align*}
Given that $\mathrm{sq}(v_{1},v_{2})$ defines a map, it follows that if
$v_{1}\equiv v_{1}^{\prime}$ $(\operatorname{mod}N)$ and $\mathrm{sq}%
(v_{1},v_{2})$ and $\mathrm{sq}(v_{1}^{\prime},v_{2}^{\prime})$ hold then
$v_{2}\equiv v_{2}^{\prime}$ $(\operatorname{mod}N)$; this map is therefore
continuous as required.

{Claim 1} is now established.

\medskip

\noindent\textbf{Claim 2. }\emph{The map} $\theta$ from $G$ \emph{to}
$G^{(4)}$ \emph{given by}
\[
g=\left(
\begin{array}
[c]{cc}%
a & b\\
c & d
\end{array}
\right)  \longmapsto\left(
\begin{array}
[c]{cc}%
u(a) & u(b)\\
u(c) & u(d)
\end{array}
\right)  \in\mathrm{M}_{2}(U)\subseteq G^{(4)}%
\]
\emph{is definable}.

To begin with, we partition $G$ as $G_{1}\overset{\cdot}{\cup}G_{2}$ where%
\begin{align*}
G_{1}  &  =\{g\in G\mid g_{11}\in R^{\ast}\}\\
G_{2}  &  =\{g\in G\mid g_{11}\in\mathfrak{m}\}.
\end{align*}

If $g\in G_{1}$ then $g=\widetilde{v}(g)\widetilde{h}(g)\widetilde{u}(g)$
where%
\begin{align*}
\widetilde{v}(g)  &  =v(-a^{-1}c)\in V\\
\widetilde{h}(g)  &  =h(a^{-1})\in H\\
\widetilde{u}(g)  &  =u(a^{-1}b)\in U.
\end{align*}
This calculation shows that in fact $G_{1}=VHU$, so $G_{1}$ is definable;
these three functions on $G_{1}$ are definable since%
\begin{align*}
x  &  =\widetilde{v}(g)\Longleftrightarrow x\in V\cap HUg\\
y  &  =\widetilde{u}(g)\Longleftrightarrow y\in U\cap HVg\\
z  &  =\widetilde{h}(g)\Longleftrightarrow z\in H\cap VgU.
\end{align*}
If $g\in G_{2}$ then $gw\in G_{1}$ since $a$ and $b$ cannot both lie in
$\mathfrak{m}$, and then%
\[
g=\widetilde{v}(gw)\widetilde{h}(gw)\widetilde{u}(gw)w^{-1}\text{.}%
\]

Define $\alpha:V\rightarrow U$ by $x\alpha=x^{w}$. Then $v(\lambda
)\alpha=u(\lambda)$.

Define $\beta:H\rightarrow U\times U$ as follows. Let $(x,y_{1},y_{2})\in
H\times U\times U$. Then $x\beta=(y_{1},y_{2})$ if and only if%
\[
\exists t\in H.\left(  \left(  x=t^{2}\wedge y_{1}^{t}=u\wedge y_{2}%
=u^{t}\right)  \vee\left(  x=t^{2}h\wedge y_{1}^{th}=u^{\prime}\wedge
y_{2}=u^{\prime t}\right)  \right)  .
\]
This decodes as $h(\lambda)\beta=(u(\lambda^{-1}),~u(\lambda)).$

Now we construe $\mathrm{M}_{2}(U)$ as a ring by transferring the ring
operations from $R$ to $U$ via Claim 1. As $\theta$ is a group homomorphism,
we find that for $g\in G_{1}$,%
\begin{equation}
g\theta=\left(
\begin{array}
[c]{cc}%
1 & 0\\
-\widetilde{v}(g)\alpha & 1
\end{array}
\right)  \cdot\left(
\begin{array}
[c]{cc}%
\widetilde{h}(g)\beta_{1} & 0\\
0 & \widetilde{h}(g)\beta_{2}%
\end{array}
\right)  \cdot\left(
\begin{array}
[c]{cc}%
1 & \widetilde{u}(g)\\
0 & 1
\end{array}
\right)  ; \label{picture}%
\end{equation}
while if $g\in G_{2}$ then%
\[
g\theta=(gw)\theta\cdot w^{-1}.
\]
As matrix multiplication in $\mathrm{M}_{2}(U)$ is $L_{\mathrm{gp}}$ definable
in $G^{(4)},$ it follows that $\theta$ is definable.

This completes the proof of Claim 2.

\medskip

\noindent\textbf{Claim 2}bis. \emph{The map sending} $\lambda\in R$ to
$(1,\lambda,0,1)\in R^{(4)}$ \emph{is definable. }Obviously.

\medskip

\noindent\textbf{Claim 3}. \ $G$ is algebraically rigid because it is strongly
complete. This is presumably well known, but we include a proof in Proposition
\ref{linsc} below.

\medskip

The theorem follows.
\end{proof}

\begin{lemma}
\label{subgp}Let $G$ be a group and $\rho$ a formula such that $\rho(G)$ is a
subgroup. Then there is a formula $\rho^{\ast}$ such that $\rho^{\ast}%
(G)=\rho(G),$ and for any group $H$ the subset $\rho^{\ast}(H)$ is a subgroup.
If $\rho$ defines a closed subset in every profinite group then so does
$\rho^{\ast}$.
\end{lemma}

\begin{proof}
Put $\rho_{1}(x)=\rho(x)\vee(x=1),$ and set%
\[
\rho^{\ast}(x):=\rho_{1}(x)\wedge\forall y\left(  \rho_{1}(y)\longrightarrow
(\rho_{1}(xy)\wedge\rho_{1}(x^{-1}y))\right)  .
\]
One sees easily that this has the required properties.
\end{proof}

\begin{proposition}
\label{linsc} Let $R$ be a complete local domain, with maximal ideal
$\mathfrak{m}$ and finite residue field $R/\mathfrak{m}$ of odd
characteristic. Then the profinite group $\mathrm{SL}_{2}(R)$ is finitely
generated, hence strongly complete.
\end{proposition}

\begin{proof}
Let $G(1)$ denote the principal congruence subgroup modulo $\mathfrak{m}$ in
$G=\mathrm{SL}_{2}(R)$. Then $G(1)$ is a so-called $R$\emph{-perfect group}:
that is, an $R$-analytic pro-$p$ group whose associated Lie algebra is perfect
(see \cite{DDMS}, Exercise 13.10 on page 352). Now Corollary 3.4 of \cite{LS}
asserts that such pro-$p$ groups are finitely generated (cf.\ \cite{DDMS}%
\textbf{, }Proposition 13.29(i)), hence strongly complete (\cite{DDMS},
Theorem 1.17).

As $G(1)$ is open in $G$ it follows that $G$ has both properties.
\end{proof}

\medskip

\textbf{Remarks.} (i) The Lubotzky-Shalev results \cite{LS} hold for
$\mathrm{SL}_{n}(R)$, any $n\geq2$ (also when $\mathrm{char}(R/\mathfrak{m}%
)=2$ provided $n\geq3$).

\noindent(ii) The strong completeness can also be inferred directly from the
fact that these groups have the \emph{congruence subgroup property}: every
subgroup of finite index contains a\ principal congruence subgroup modulo
$\mathfrak{m}^{k}$ for some $k$ (this follows from \cite{K}, Satz 2).

\section{Profinite groups of finite rank\label{frcpisec}}

In the context of profinite groups, a `generating set' will always mean a
\emph{topological} generating set. From a first-order point of view, the
nicest profinite groups are the finitely generated pro-$p$ groups, where $p$
is a prime. These have the following special property: a finite generating set
can be recognized in a definable finite quotient (as we shall see below). So
for such groups, being generated by $d$ elements is a first-order property.
This is the key to most of our main results; it does not hold for
f.g.\ profinite groups in general as we point out in Prop.\ \ref{prop:neg gen}%
,
but it does hold for groups in the larger class $\mathcal{C}_{\pi}$ of
\emph{pronilpotent pro-}$\pi$ \emph{groups}, where $\pi$ is any \emph{finite}
set of primes: a group $G$ is in $\mathcal{C}_{\pi}$ iff it is a direct
product%
\[
G=G_{p_{1}}\times\cdots\times G_{p_{k}}%
\]
where $\pi=\{p_{1},\ldots,p_{k}\}$ and $G_{p_{i}}$ is a pro-$p_{i}$ group for
each $i$. To save repetition, we make the convention that $\pi$ \emph{will
always denote a finite set of primes}.

\subsection{Some preliminaries, and an example\label{prelimsec}}

For basic facts about profinite groups, see \cite{DDMS}, Chapter 1 and the
earlier chapters of \cite{WP}. Besides the language $L_{\mathrm{gp}}$ of group
theory, we will consider the language

\begin{itemize}
\item $L_{\pi}:$ the language $L_{\mathrm{gp}}$ augmented by unary function
symbols $P_{\lambda},$ one for each $\lambda\in\mathbb{Z}_{\pi}=\prod_{p\in
\pi}\mathbb{Z}_{p}$; for $g\in G$, $P_{\lambda}(g)$ is interpreted as the
profinite power $g^{\lambda}$.
\end{itemize}

For a profinite group $G,$

\begin{itemize}
\item $\mathrm{d}(G)$ is the minimal size of a (topological) generating set
for $G$

\item $\mathrm{rk}(G)=\sup\left\{  \mathrm{d}(H)\mid H\leq_{c}G\right\}  $
\end{itemize}

This is the \emph{rank} of $G$ (sometimes called \emph{Pr\"{u}fer rank}). The
pro-$p$ groups of finite rank are of particular interest, being just those
that are $p$\emph{-adic analytic}; \ see \cite{DDMS}, Chapter 3 (which
includes several equivalent definitions of rank). On the face of it, having a
particular finite rank is not a first-order property (the definition involves
quantifying over subgroups); the following result shows that the rank, if
finite, can be more or less specified by a first-order sentence:

\begin{proposition}
For each positive integer $r$, there is a sentence $\rho_{r}$ such that for a
pro-$p$ group $G$,%
\[
\mathrm{rk}(G)\leq r\Longrightarrow G\models\rho_{r}~\Longrightarrow
\mathrm{rk}(G)\leq r(2+\log_{2}(r)).
\]

\end{proposition}

\noindent We omit the proof, an application of the techniques described below.

\medskip

We fix the notation%
\begin{align*}
q(\pi)  &  =p_{1}\ldots p_{k}\\
q^{\prime}(\pi)  &  =2^{\varepsilon}q(\pi)
\end{align*}
where $\varepsilon=0$ if $2\notin\pi,$ $\varepsilon=1$ if $2\in\pi.$

\bigskip

A sharper version of Theorem \ref{NSSC}(ii) holds in some cases:

\begin{lemma}
\label{dg}Let $G=\overline{\left\langle a_{1},\ldots,a_{d}\right\rangle }$ be
a pronilpotent group. Then%
\begin{equation}
G^{\prime}=[a_{1},G]\ldots\lbrack a_{d},G], \label{1}%
\end{equation}
a closed subgroup of $G$.
\end{lemma}

\begin{proof}
The set $X=[a_{1},G]\ldots\lbrack a_{d},G]$ is closed in $G,$ so
$X=\bigcap_{N\vartriangleleft_{o}G}XN$. If $N\vartriangleleft_{o}G$ then $G/N$
is nilpotent and generated by $\{a_{1},\ldots,a_{d}\}$, which implies
$XN/N=G^{\prime}N/N$ (cf. \cite{DDMS} Lemma 1.23), so $G^{\prime}\leq XN$.
Hence $G^{\prime}\subseteq X.$
\end{proof}

It is easy to see that (\ref{1}) can be expressed by a first-order formula. If
in a group~$G$ every product of $d+1$ commutators belongs to the set $X$
defined above, then (by an obvious induction) every product of commutators
belongs to $X$. Hence there is a formula $\alpha$ such that%
\[
G\models\alpha(a_{1},\ldots,a_{d})\Longleftrightarrow G^{\prime}%
=[a_{1},G]\ldots\lbrack a_{d},G].
\]

The \emph{Frattini subgroup} $\Phi(G)$ of a profinite group $G$ is the
intersection of all maximal open subgroups of $G$. It follows from the
definition that for $Y\subseteq G$,%
\begin{equation}
\overline{\left\langle Y\right\rangle }=G\Longleftrightarrow\overline
{\left\langle Y\right\rangle }\Phi(G)=G. \label{2}%
\end{equation}
If $G$ is pronilpotent then every maximal open subgroup is normal of prime
index; it follows that if $G\in\mathcal{C}_{\pi}$ then%
\[
\Phi(G)=\overline{G^{\prime}}G^{q(\pi)}.
\]
If also $G$ is finitely generated then this subgroup has finite index, so it
is \emph{open}$,$ and then in (\ref{2}) we have%
\[
\overline{\left\langle Y\right\rangle }\Phi(G)=\left\langle Y\right\rangle
\Phi(G)=\left\langle Y\right\rangle G^{\prime}G^{q(\pi)}.
\]

Thus if $G=\overline{\left\langle a_{1},\ldots,a_{d}\right\rangle }%
\in\mathcal{C}_{\pi}$ then%
\[
\Phi(G)=\delta(\overline{a};G)
\]

where%
\[
\delta(\overline{a},x)\equiv\exists z,y_{1},\ldots,y_{d}.~x=[a_{1}%
,y_{1}]\ldots\lbrack a_{d},y_{d}]z^{q(\pi)}.
\]

The definability of $\Phi(G)$ means that we can define generating sets in
$L_{\mathrm{gp}}$:

\begin{proposition}
\label{defgen}For each $d\geq1$ there is a formula $\beta_{d}$ such that for
$G\in\mathcal{C}_{\pi}$,
\[
G\models\beta_{d}(a_{1},\ldots,a_{d})~\Longleftrightarrow~G=\overline
{\left\langle a_{1},\ldots,a_{d}\right\rangle }.
\]

\end{proposition}

\begin{proof}
Set%
\[
\beta_{d}(u_{1},\ldots,u_{d})\equiv\alpha(u_{1},\ldots,u_{d})\wedge\forall
x.~\bigvee_{s(1),\ldots,s\left(  d\right)  \in S}\delta(\overline{u}%
,x^{-1}u_{1}^{s(1)}\ldots u_{d}^{s(d)})
\]
where $S=\{0,1,\ldots,q-1\}$. We have seen that if $G=\overline{\left\langle
a_{1},\ldots,a_{d}\right\rangle }\in\mathcal{C}_{\pi}$ then $G\models\beta
_{d}(a_{1},\ldots,a_{d})$.

Conversely, if $G\in\mathcal{C}_{\pi}$ and $G\models\beta_{d}(a_{1}%
,\ldots,a_{d})$ then every element of $G$ belongs to
\[
\left\langle a_{1},\ldots,a_{d}\right\rangle [a_{1},G]\ldots\lbrack
a_{d},G]G^{q(\pi)}\subseteq\left\langle a_{1},\ldots,a_{d}\right\rangle
\Phi(G).
\]

\end{proof}

\medskip

Now setting%
\[
\widetilde{\beta}_{d}\equiv\exists y_{1},\ldots,y_{d}.\beta_{d}(y_{1}%
,\ldots,y_{d})
\]

we have%
\begin{align*}
\mathrm{d}(G)\leq d~\Longleftrightarrow~  &  G\models\widetilde{\beta}_{d},\\
\mathrm{d}(G)=d~\Longleftrightarrow~  &  G\models\widetilde{\beta}_{d}%
\wedge\lnot\widetilde{\beta}_{d-1}:=\beta_{d}^{\ast};
\end{align*}
thus for groups in $\mathcal{C}_{\pi}$ being $d$-generated can be expressed by
a first-order sentence.\medskip

We note that the hypothesis that the group $G$ be in $\mathcal{C}_{\pi}$ is necessary:

\begin{proposition}
\label{prop:neg gen} Within profinite groups, being $d$-generated cannot be
expressed by a single first order sentence.
\end{proposition}

\begin{proof}
According to Proposition \ref{neg1}, if $\phi$ is a sentence in the language
of groups and a non-periodic abelian group $G$ satisfies $\phi$ then $G\times
C_{q}$ satisfies $\phi$ for almost all primes $q$.

If $\phi$ expresses being $d$-generated, then $\widehat{\mathbb{Z}}^{d}%
\models\phi$. Let $q$ be a prime as above, then also $\widehat{\mathbb{Z}}%
^{d}\times C_{q}\models\phi$. But this group needs $d+1$ generators.
\end{proof}

\medskip

The same argument works for the category of abstract groups, using
$\mathbb{Z}^{d}$ in place of $\widehat{\mathbb{Z}}^{d}$. A slightly more
elaborate argument shows that being finitely generated, for profinite groups,
also cannot be expressed by a single first order sentence $\phi$. One works
with $\widehat{\mathbb{Z}}\times(C_{q})^{\aleph_{0}}$, and uses the fact that
$\phi$ can be expressed as a Boolean combination of Szmielew invariant
sentences: see \cite{HMT}, Thm.\ A.2.7.

\medskip

To conclude this introductory section, we discuss a `small' f.g.\ pro-$p$
group of infinite rank,%

\[
G:=C_{p}\widehat{\wr}\,\mathbb{Z}_{p}=\,\underleftarrow{\lim}\,_{n\rightarrow
\infty}C_{p}\wr C_{p^{n}}.
\]
This group is the semidirect product of the `base group' $M$ by a procyclic
group $T\cong\mathbb{Z}_{p}$; here
\[
M\cong\mathbb{F}_{p}[[T]]
\]
as a $T$-module, where $\mathbb{F}_{p}[[T]]$ is the completed group algebra of
$T$ (see \cite{DDMS}, \S 7.4). Note that $\mathbb{F}_{p}[[T]]$ is a
$1$-dimensional complete local ring with residue field~$\mathbb{F}_{p}$, whose
non-zero closed ideals are just the powers of the maximal ideal, and therefore
have finite index.

\begin{proposition}
\label{wr}The pro-$p$ group $C_{p}\widehat{\wr}\,\mathbb{Z}_{p}$ is FA within
profinite groups.
\end{proposition}

\begin{proof}
Let $a$ be a generator of the $T$-module $M,$ and let $f$ be a generator of
$T$. Then $G$ has the pro-$p$ presentation%
\[
\left\langle a,f;~a^{p}=[w,a]=[u,w]=1~(u,w\in F^{\prime})\right\rangle
\]
where $F$ denotes the free group on $\{a,f\}$.

Let $\sigma(a,f)$ be the formula saying for a pro-$p$ group $G$ that

\begin{itemize}
\item the elements $a,f$ generate $G$, i.e. $\beta_{2}(a,f)$ holds;

\item the element $a$ has order $p$ and commutes with every commutator;

\item all commutators commute;

\item the centre of $G$ is trivial.
\end{itemize}

Then $\sigma(a,f)$ holds in $G$.

Suppose to begin with that $H$ is a pro-$p$ group and that $H\models
\sigma(b,h)$ for some $b,$ $h\in H$. Then the map sending $a$ to $b$ and $f$
to $h$ extends to an epimorphism $\theta:G\rightarrow H$. Let $N=\ker\theta$.
We aim to show that $N=1$.

Now $G/N\cong H$ is a non-trivial pro-$p$ group with trivial centre, so it is
infinite (as every non-trivial finite $p$-group has non-trivial centre). Also
$N\cap M$ corresponds to a closed ideal of $\mathbb{F}_{p}[[T]],$ so if $N\cap
M\neq1$ then $M/\left(  N\cap M\right)  $ is finite. As $G/M$ is procyclic
this implies that the centre of $G/N$ has finite index (if $f^{n}$ centralizes
$M/\left(  N\cap M\right)  $ then $[G,N\overline{\left\langle f^{n}%
\right\rangle }]\leq N$ ). This contradicts $\mathrm{Z}(H)=1$, and we conclude
that $N\cap M=1$. But then $N\leq\mathrm{C}_{G}(M)\cap N=M\cap N=1$.

It follows that $H\cong G$. Thus $(G;a,f)$ is FA in pro-$p$ groups.

To deal with the general case of profinite groups, we need a way to identify
the prime $p$. Now we have%
\begin{align*}
G^{\prime}  &  =[M,f]=\{[x,y]~\mid~x,y\in G\}\\
M  &  =G^{\prime}\left\langle a\right\rangle =G^{\prime}\cup G^{\prime}%
a\cup\ldots\cup G^{\prime}a^{p-1}%
\end{align*}
so $M$ is definable by a formula $\mu(a,f)$ say. Since $G$ is pro-$p$,
$\mathrm{C}_{G}(M)=M$ and $M^{p}=1,$ the following holds:%
\[
\lbrack M,x]\subseteq\lbrack M,x^{p}]\Longrightarrow\lbrack
M,x]=1\Longrightarrow x\in M\Longrightarrow x^{p}=1.
\]
So $G$ satisfies a formula $\tau(a,f)$ that expresses
\[
M\lhd G\text{ and }[M,M]=M^{p}=1\text{ and }[M,x]\subseteq\lbrack
M,x^{p}]\Longrightarrow x^{p}=1.
\]

Suppose now that $H$ is a profinite group and that $H\models\sigma
(b,h)\wedge\tau(b,h).$ Let $N=\mu(b,h;H),$ so $N$ is an abelian normal
subgroup of $H,$ of exponent $p$. Suppose $x$ belongs to a Sylow pro-$q$
subgroup of $H$ where $q\neq p$. Then $x=x^{p\lambda}$ for some $\lambda
\in\mathbb{Z}_{q}$, and then $[N,x]=[N,x^{p\lambda}]\subseteq\lbrack N,x^{p}%
]$. As $H\models\tau(b,h)$ this implies that $x^{p}=1,$ and hence that $x=1$.
It follows that $H$ is a pro-$p$ group, and then $H\models\sigma(b,h)$ implies
that $H\cong G$.
\end{proof}

\begin{corollary}
The classes of profinite, respectively pro-$p$, groups of finite rank and of
f.g.\ profinite, respectively pro-$p$, groups are first-order separable, with
witness group $C_{p}\widehat{\wr}\,\mathbb{Z}_{p}$.
\end{corollary}

\subsection{Powerful pro-$p$ groups}

\label{powsec}

Next we discuss a special class of pro-$p$ groups, where $p$ always denotes a
prime. Fix%
\[
\varepsilon=0~\text{if }p\neq2,~\varepsilon=1~\text{if }p=2.
\]
A pro-$p$ group $G$ is \emph{powerful }if $G/\overline{G^{p}}$ is abelian
(replace $p$ by $4$ when $p=2$). If $G$ is also finitely generated, then%
\[
\overline{G^{p^{n}}}=G^{p^{n}}=G^{\{p^{n}\}}%
\]
for each $n\geq1$. Thus for a f.g.\ pro-$p$ group $G,$ $G$ is powerful iff%
\[
G\models\forall x,y\exists z.\left(  [x,y]=z^{p}\right)
\]
(replace $p$ by $4$ when $p=2$).\ In this case we have%
\[
\Phi(G)=G^{p}~~\text{(resp. }G^{4}\text{ if }p=2\text{).}%
\]

For all this, see \cite{DDMS}, Chapter 3. A key result of Lazard
\cite{Lz}\ characterizes the compact $p$-adic analytic groups as the f.g.
profinite groups that are virtually powerful (cf.\ \cite{DDMS}, Chapter 8).

The definition of a \emph{uniform} pro-$p$ group is given in \cite{DDMS},
Chapter~4. Rather than repeating it here, we use the simple characterization
(\emph{loc.\ cit.} Theorem 4.5): a pro-$p$ group is uniform iff it is
\emph{f.g., powerful }and\emph{\ torsion-free}.

The following theorem summarizes key facts established in Chapters 3 and 4 of
\cite{DDMS}. We set%
\[
\lambda(r)=\left\lceil \log_{2}r\right\rceil +\varepsilon.
\]

\begin{theorem}
\label{pow}\emph{(i) }Let $G$ be a pro-$p$ group of finite rank $r$. Put
$m=\lambda(r)$. Then $G$ has open normal subgroups $W\geq W_{0},$ with%
\[
W\geq\Phi^{m}G\geq G^{p^{m}},
\]
such that every open normal subgroup of $G$ contained in $W$ is powerful, and
every open normal subgroup of $G$ contained in $W_{0}$ is uniform.

\emph{(ii)} If $G$ is f.g.\ and powerful then $\mathrm{rk}(G)=\mathrm{d}(G)$.

\emph{(iii) }If $G=\overline{\left\langle a_{1},\ldots,a_{d}\right\rangle }$
is powerful then%
\[
G=\left\{  a_{1}^{\mu_{1}}\ldots a_{d}^{\mu_{d}}\mid\mu_{1},\ldots,\mu_{d}%
\in\mathbb{Z}_{p}\right\}  .
\]

\emph{(iv)} If $G$ is f.g.\ and powerful then $G$ has a uniform open normal
subgroup $U$, and
\[
\mathrm{d}(V)=\mathrm{rk}(V)=\mathrm{rk}(U)=\mathrm{d}(U)
\]
for every uniform open subgroup $V$ of $G$.
\end{theorem}

The common rank of open uniform subgroups of such a group $G$ is denoted
$\dim(G)$; this is the dimension of $G$ as a $p$-adic analytic group.

\begin{lemma}
\label{add-dim}Let $G$ be a uniform pro-$p$ group and let $N\vartriangleleft
_{c}G$. If $G/N$ is uniform then $N$ is uniform and%
\[
\dim(G)=\dim(N)+\dim(G/N).
\]

\end{lemma}

\begin{proof}
As explained in \cite{DDMS}, Chapter 4, $G$ has the structure of a
$\mathbb{Z}_{p}$-Lie algebra $L(G)$, additively isomorphic to $\mathbb{Z}%
_{p}^{\dim(G)} $. Proposition 4.31 of \cite{DDMS} says that $N$ is uniform,
$L(N)$ is an ideal of $L(G)$, and the quotient mapping $G\rightarrow G/N$
induces an epimorphism $L(G)\rightarrow L(G/N)$. The claim follows from the
additivity of dimension for free $\mathbb{Z}_{p}$-modules.
\end{proof}

\begin{corollary}
\label{dims}Let $G$ be a pro-$p$ group of finite rank and let
$N\vartriangleleft_{c}G$. Then $\dim(G)=\dim(N)+\dim(G/N)$.
\end{corollary}

\begin{proof}
Let $H$ be a uniform open normal subgroup of $G$. Then $H/(H\cap N)$ is
powerful, hence has a finite normal subgroup $M/N$ such that $H/M$ is uniform
by~\cite{DDMS}, Theorem 4.20. The claim follows on replacing $G$ by $H$ and
$N$ by~$M$.
\end{proof}

These results can be applied to $\mathcal{C}_{\pi}$ groups of finite rank. Let
$G\in\mathcal{C}_{\pi}$. Then $G=G_{1}\times\cdots\times G_{k}$ where each
$G_{t}$ is a pro-$p_{t}$ group, the Sylow pro-$p_{t}$ subgroup of~$G$. If $H$
is a closed subgroup of $G$ then $H=H_{1}\times\cdots\times H_{k}$ where
$H_{t}=H\cap G_{t}$, notation we keep for the remainder of this subsection.

If $G$ has finite rank, we define
\[
\mathrm{Dim}(G)=\dim G_{1}+\cdots+\dim G_{k}.
\]

If $p_{t}\nmid m$ then every element of $H_{t}$ is an $m$th power in $H$; thus
if $q=p_{1}^{e_{1}}\ldots p_{k}^{e_{k}}$ then
\begin{align*}
H^{\{q\}}  &  =H_{1}^{\{p_{1}^{e_{1}}\}}\times\cdots\times H_{k}%
^{\{p_{k}^{e_{k}}\}}\\
H^{q}  &  =H_{1}^{p_{1}^{e_{1}}}\times\cdots\times H_{k}^{p_{k}^{e_{k}}}.
\end{align*}
We call $H$ \emph{semi-powerful }if each $H_{t}$ is a powerful pro-$p_{t}$
group. If $H\in\mathcal{C}_{\pi}$ is finitely generated, then $H$ is
semi-powerful if and only if $H/H^{q^{\prime}(\pi)}$ is abelian. This holds
iff%
\[
H\models\mathrm{pow}\equiv\forall x,y\exists z.\left(  [x,y]=z^{q^{\prime}%
(\pi)}\right)  .
\]

$H$ is \emph{semi-uniform} if each $H_{t}$ is uniform. In this case, the
\emph{dimension} of $H$ is the $k$-tuple%
\[
\dim H=(\dim H_{1},\ldots,\dim H_{k}).
\]

\begin{lemma}
\label{epi}Let $H$ and $K$ be semi-uniform $\mathcal{C}_{\pi}$ groups, and
$\theta:H\rightarrow K$ an epimorphism. If $\dim H=\dim K$ then $\theta$ is an isomorphism.
\end{lemma}

\begin{proof}
Restricting to Sylow subgroups, we may suppose that $H$ and $K$ are uniform
pro-$p$ groups of the same dimension. Then Lemma \ref{add-dim} shows that
$\ker\theta$ is a uniform group of dimension $0$, i.e. the trivial group.
\end{proof}

\begin{corollary}
\label{finnorm}Let $G\in\mathcal{C}_{\pi}$ have finite rank. If
$N\vartriangleleft_{c}G$ and $\mathrm{Dim}(G/N)=\mathrm{Dim}(G)$ then $N$ is finite.
\end{corollary}

\begin{proof}
This follows likewise from Corollary \ref{dims}.
\end{proof}

\medskip

For $q,~f\in\mathbb{N}$ set%
\[
\mu_{f,q}(x)\equiv\exists y_{1},\ldots,y_{f}.(x=y_{1}^{q}\ldots y_{f}^{q}).
\]
As before, we see that the word $x^{q}$ has width $f$ in a group $H,$ that is,%
\[
H^{q}=(H^{\{q\}})^{\ast f}:=\left\{  h_{1}^{q}\ldots h_{f}^{q}\mid h_{i}\in
H\right\}  ,
\]
if and only if $H$ satisfies%
\[
\mathrm{m}_{f,q}\equiv\forall x.\left(  \mu_{f+1,q}(x)\rightarrow\mu
_{f,q}(x)\right)  .
\]
Of course, this holds iff $H\models\mathrm{s}(\mu_{f,q})$; we can use either formulation.

\begin{proposition}
\label{powpi}Let $G$ be a f.g.\ profinite group and let $q\in\mathbb{N}$.
There exists $f\in\mathbb{N}$ such that $G\models\mathrm{m}_{f,q}$ and%
\[
G^{q}=(G^{\{q\}})^{\ast f}=\mu_{f,q}(G)
\]
is a definable open normal subgroup of $G$. Both $f$ and $\left\vert
G:G^{q}\right\vert $ can be bounded in terms of $q$ and $\mathrm{d}(G)$.
\end{proposition}

\noindent This is part of Theorem \ref{NSSC}; the second claim was not made
explicit in the statement but is included in the proof.

\medskip

If $H$ is semi-uniform of dimension $(d_{1},\ldots,d_{k})$, we have
\begin{align*}
\Phi(H)  &  =H^{q(\pi)}=H^{\{q(\pi)\}}=\mu_{1,q(\pi)}(H),\\
\left\vert H:\Phi(H)\right\vert  &  =p_{1}^{d_{1}}\ldots p_{k}^{d_{k}}.
\end{align*}
Thus for semi-uniform $H$, the dimension is determined by%
\[
H\models\partial_{d_{1},\ldots,d_{k}}\equiv\mathrm{ind}(\mu_{1,q(\pi)}%
;p_{1}^{d_{1}}\ldots p_{k}^{d_{k}}).
\]

\textbf{Remark.} The primary components of a $\mathcal{C}_{\pi}$ group are not
in general definable; this can be deduced from \cite{TZ}, Theorem 3.3.5. If
they were, our results about $\mathcal{C}_{\pi}$ could be reduced to the
slightly less messy case where $\pi$ consists of a single prime; it is easy to
see that a direct product of finitely many definable subgroups is FA if each
of the factors is FA.

For f.g. groups with trivial centre such an approach is feasible, as each
primary component is then a centralizer, and so definable (with parameters). A
slightly more elaborate argument, using \cite{NS}, Theorem 1.6, shows that the
same holds for f.g. groups with finite abelianization.

\subsection{Presentations}

\label{present}

In the context of profinite groups, a `finite presentation' may involve
relators that are `profinite words', i.e. limits of a convergent sequence of
group words. For present purposes we need to consider concepts of finite
presentation that are both more and less restrictive.

Let $\mathcal{C}$ be a class of groups and $L\supseteq L_{\mathrm{gp}}$ a
language. For a group $G\in\mathcal{C}$ and a formula $\psi(x_{1},\ldots
,x_{r})$ of $L$, we say that $\psi$ is an $L$\emph{-presentation} of $G$ in
$\mathcal{C}$ if $G$ has a generating set $\{g_{1},\ldots,g_{r}\}$ such that

\begin{itemize}
\item[(i)] $G\models\psi(g_{1},\ldots,g_{r})$, and

\item[(ii)] if $h_{1},\ldots,h_{r}\in H\in\mathcal{C}$ and $H\models\psi
(h_{1},\ldots,h_{r})$ then there is an epimorphism $\theta:G\rightarrow H$
with $g_{i}\theta=h_{i}$ for each $i$.
\end{itemize}

In this case, we say that $\psi$ is an $L$-presentation on $\{g_{1}%
,\ldots,g_{r}\}$.

The concept of $L$-presentation generalizes the familiar idea of a finite
presentation in group theory. We mention two particular cases.

\begin{proposition}
\label{L-pres}A group $G\in\mathcal{C}_{\pi}$ has an $L$-presentation in
$\mathcal{C}_{\pi}$ in each of the following cases:

\medskip

\noindent\emph{(i)} $L=L_{\mathrm{gp}}$, and $G$ is \emph{strictly f.p.} in
$\mathcal{C}_{\pi}$; that is, $G$ has a finite presentation as a
$\mathcal{C}_{\pi}$-group in which the relators are finite group words, or
equivalently, $G$ is the $\mathcal{C}_{\pi}$-completion of a finitely
presented (abstract) group.

\medskip

\noindent\emph{(ii) }$L=L_{\pi}$, and $G$ has finite rank.
\end{proposition}

\begin{proof}
(i) We have an epimorphism $\phi:F\rightarrow G$ where $F$ is the free
$\mathcal{C}_{\pi}$-group on a finite generating set $X=\{x_{1},\ldots
,x_{r}\}$ and $\ker\phi$ is the closed normal subgroup of $F$ generated by a
finite set $R$ of ordinary group words on $X$. Set%
\[
\psi(\overline{x}):=\beta_{r}(\overline{x})\wedge\bigwedge_{w\in R}%
w(\overline{x})=1
\]
(recall that $G\in\mathcal{C}_{\pi}$ satisfies $\beta_{r}(\overline{a})$ iff
$\{a_{1},\ldots,a_{r}\}$ generates $G$).

Now put $g_{i}=x_{i}\phi$ for each $i$. Then $G\models\psi(\overline{g})$.
Suppose that $h_{1},\ldots,h_{r}\in H\in\mathcal{C}_{\pi}$ and $H\models
\psi(\overline{h})$. Then $h_{1},\ldots,h_{r}$ generate $H$, so the
homomorphism $\mu:F\rightarrow H$ sending $\overline{x}$ to $\overline{h}$ is
onto. Also for each $w\in R$ we have $w(\overline{x})\mu=w(\overline{h})=1 $,
so $\ker\phi\leq\ker\mu$. It follows that $\mu$ factors through an epimorphism
$\theta:G\rightarrow H$ with $g_{i}\theta=h_{i}$ for each $i$. Thus $\psi$ is
an $L_{\mathrm{gp}}$ presentation for $G$ in $\mathcal{C}_{\pi}$.
\end{proof}

\bigskip

Before proving (ii) we need yet another definition:

\begin{itemize}
\item Let $G\in\mathcal{C}_{\pi}$. Then $\left\langle X;R\right\rangle $ is a
$\mathbb{Z}_{\pi}$\emph{-finite} $\mathcal{C}_{\pi}$ presentation for $G$ if
$G\cong F/N$ where $F=F(X)$ is the free $\mathcal{C}_{\pi}$-group on a finite
generating set $X$ and $N$ is the closed normal subgroup of $F$ generated by a
finite set $R$ of elements of the form%
\begin{equation}
w_{1}^{\mu_{1}}\ldots w_{n}^{\mu_{n}} \label{piword}%
\end{equation}
where each $w_{i}$ is a group word on $X$ and $\mu_{1},\ldots,\mu_{n}%
\in\mathbb{Z}_{\pi}$.
\end{itemize}

Note that the free $\mathcal{C}_{\pi}$-group on a set $X$ is the direct
product of its Sylow pro-$p_{t}$ subgroups, which themselves are free
pro-$p_{t}$ groups on the $p_{t}$-components of $X$.


Expressions like (\ref{piword}) will be called $\pi$\emph{-words. }For a
subset\emph{\ }$Y$ of $G$, one says that $\left\langle X;R\right\rangle $ is a
\emph{presentation on} $Y$ if the implied epimorphism $F(X)\rightarrow G$ maps
$X$ to~$Y$.

\begin{lemma}
\label{sfpres}Let $G=\overline{\left\langle Y\right\rangle }\in\mathcal{C}%
_{\pi}$ have finite rank, where $Y$ is finite. Then $G$ has a $\mathbb{Z}%
_{\pi}$-finite $\mathcal{C}_{\pi}$ presentation on $Y$.
\end{lemma}

\begin{proof}
Suppose $U$ is a uniform pro-$p$ group with generating set $X=\{x_{1}%
,\ldots,x_{d}\}$. For simplicity, assume that $\{x_{1},\ldots,x_{r}\}$ is a
basis for $U$, where $r=\dim U\leq d$. Then $U$ has a pro-$p$ presentation on
$X$ with relators%
\begin{align}
\lbrack x_{i},x_{j}]x_{1}^{\lambda_{1}(i,j)}\ldots x_{r}^{\lambda_{r}%
(i,j)},~~1  &  \leq i<j\leq r,\label{unifpres}\\
x_{k}x_{1}^{\mu_{1}(k)}\ldots x_{r}^{\mu_{r}(k)},~~r  &  <k\leq d,
\label{unifpres2}%
\end{align}
with $\lambda_{l}(i,j)\in p\mathbb{Z}_{p}$ and $\mu_{l}(k)\in\mathbb{Z}_{p}$;
Prop.\ 4.32 of \cite{DDMS} gives the presentation (\ref{unifpres}), and
(\ref{unifpres2}) expresses the redundant generators using using Theorem
\ref{pow}(iii).

We rewrite (\ref{unifpres2}) (also adding some trivial relators) as%
\begin{equation}
x_{1}^{\mu_{1}(k)}\ldots x_{d}^{\mu_{d}(k)},~~1\leq k\leq d,
\label{uniffpres3}%
\end{equation}
by setting $\mu_{l}(k)=\delta_{kl}$ for $k,l>r$ and $\mu_{l}(k)=0$ for $k<r$
and all $l$. Next, we add redundant relators of the form (\ref{unifpres}) for
all $j>r$ and all $i<j,$ again using Theorem \ref{pow}(iii). Setting
$\lambda_{l}(i,j)=0$ whenever $l>r,$ we obtain a presentation for $U$ on $X$
with relators (\ref{uniffpres3}) together with%
\begin{equation}
\lbrack x_{i},x_{j}]x_{1}^{\lambda_{1}(i,j)}\ldots x_{d}^{\lambda_{d}%
(i,j)},~~1\leq i<j\leq d. \label{unifpres4}%
\end{equation}

Now consider a semi-uniform $\mathcal{C}_{\pi}$ group $V=U_{1}\times
\cdots\times U_{k}$ where each $U_{t}$ is a uniform pro-$p_{t}$ group of
dimension $d_{t}\leq d$. A generating set $X=\{x_{1},\ldots,x_{d}\}$ for $V$
projects to a generating set $X_{t}=\{x_{1}^{(t)},\ldots,x_{d}^{(t)}\}$ for
$U_{t}$; then $U_{t}$ has a pro-$p_{t}$ presentation on $X_{t}$ like
(\ref{uniffpres3})$\cup$(\ref{unifpres4}), with exponents $\lambda_{l}%
^{(t)}(i,j),~\mu_{l}^{(t)}(k)\in\mathbb{Z}_{p_{t}}$.

Let $\lambda_{l}(i,j),~\mu_{l}(k)\in\mathbb{Z}_{\pi}$ have $\mathbb{Z}_{p_{t}%
}$-components $\lambda_{l}^{(t)}(i,j)$,~$\mu_{l}^{(t)}(k)$ respectively for
each $t$. Then (\ref{uniffpres3})$\cup$(\ref{unifpres4}) gives a presentation
for $V$ on $X$.

Finally, we have $G=G_{1}\times\cdots\times G_{k}=\overline{\left\langle
y_{1},\ldots,y_{m}\right\rangle }$ where each $G_{t}$ is a pro-$p_{t}$ group
of finite rank and $Y=\{y_{1},\ldots,y_{m}\}$. Let $V=U_{1}\times\cdots\times
U_{k}$ be a semi-uniform open normal subgroup of $G$. Using the Schreier
process we obtain a finite generating set $X=\{x_{1},\ldots,x_{d}\}$ for $V,$
each element of $X$ being equal to a finite word on $\overline{y}$, say
$x_{i}=w_{i}(\overline{y})$. Substitute $w_{i}(\overline{y})$ for $x_{i}$ in
(\ref{uniffpres3})$\cup$(\ref{unifpres4}) to obtain a set of relators $R$ on
$\overline{y}$. Note that $R$ consists of $\pi$-words.

By Theorem \ref{pow}(iii), each element of $V$ is a finite product of
$\mathbb{Z}_{\pi}$-powers of elements of $X$. The conjugation action of $G$ on
$V$ is determined by specifying, for $j=1,\ldots,m$ and for each $x_{i}\in X,$%
\[
y_{j}^{-1}x_{i}y_{j}=W_{ij}(X)
\]
where each $W_{ij}(X)$ is a finite product of $\mathbb{Z}_{\pi}$- powers (for
clarity, we keep $w$ for finite group words and write $W$ for $\pi$-words).

Let%
\[
S:=\left\{  y_{j}^{-1}w_{i}^{-1}y_{j}.W_{ij}(w_{1},\ldots,w_{d})\mid
j=1,\ldots,m,~i=1,\ldots,d\right\}  .
\]

A standard argument (see for example \cite{PG} Chapter 8, Lemma 10) now shows
that $\left\langle Y;R\cup S\right\rangle $ is a $\mathcal{C}_{\pi}$
presentation for $G$.
\end{proof}

\bigskip

Now we can give the\medskip

\textbf{Proof} of Proposition \ref{L-pres} (ii). Let $G=\overline{\left\langle
Y\right\rangle }$ be as in the preceding lemma. Set%
\[
\rho(\overline{y}):=\bigwedge_{w\in R\cup S}w(\overline{y})=1
\]
where $R$ and $S$ are given above. As these are finite sets of $\pi$-words,
$\rho$ is a formula of $L_{\pi}$. Now put%
\[
\psi(\overline{y}):=\beta_{m}(\overline{y})\wedge\rho(\overline{y}).
\]
If $h_{1},\ldots,h_{m}\in H\in\mathcal{C}_{\pi}$ and $H\models\psi
(\overline{h})$ then $h_{1},\ldots,h_{m}$ generate $H$ and satisfy the
relations $R\cup S=1$; as $\left\langle Y;R\cup S\right\rangle $ is a
$\mathcal{C}_{\pi}$ presentation for $G$ it follows that the map sending
$\overline{h}$ to $\overline{y}$ extends to an epimorphism from $G$ to $H$.
Thus $\psi$ is an $L_{\pi}$ presentation for $G$ in $\mathcal{C}_{\pi}$.

\subsection{Finite axiomatizability in $\mathcal{C}_{\pi}$}

\label{proofT1} \ 

\noindent Until further notice, $L$ stands for one of $L_{\mathrm{gp}}$,
$L_{\pi}$. Theorem \ref{frank} is included in

\begin{theorem}
\label{newT1}Suppose that $G\in\mathcal{C}_{\pi}$ has finite rank, and that
$G$ has an $L$ presentation on the generating tuple $(a_{1},\ldots,a_{r})$.
Then $(G,\overline{a})$ is FA in $\mathcal{C}_{\pi}$.
\end{theorem}

\noindent Note that when $L=L_{\pi}$, the existence of an $L$ presentation is
guaranteed by Proposition \ref{L-pres}(ii).\medskip

\begin{proof}
We have $G=G_{1}\times\cdots\times G_{k}$ where each $G_{t}$ is a pro-$p_{t}$
group of finite rank and dimension $d_{t}$. There is a formula $\psi$ of $L$
such that (i) $G\models\psi(\overline{a})$ and (ii) if $h_{1},\ldots,h_{r}\in
H\in\mathcal{C}_{\pi}$ and $H\models\psi(\overline{h})$ then there is an
epimorphism $\theta:G\rightarrow H$ sending $\overline{a}$ to $\overline{h}$.

It follows from Theorem \ref{pow} that $G$ has an open normal subgroup $W_{0}
$ such that every open normal subgroup of $G$ contained in $W_{0}$ is
semi-uniform. Then $W_{0}\geq G^{q}$ for some $\pi$-number $q$. Now
Proposition \ref{powpi} shows that for some $f$, $G$ satisfies $\mathrm{m}%
_{f,q}$ and%
\[
G^{q}=\mu_{f,q}(G)
\]
is open in $G$, hence semi-uniform. Set%
\[
\mathrm{tf}\equiv\forall x(x^{q(\pi)}=1\rightarrow x=1).
\]
As $G^{q}$ is semi-powerful and torsion-free, $G$ satisfies%
\[
\mathrm{res}(\mu_{f,q},\mathrm{pow}\wedge\mathrm{tf})
\]
(see \S \ref{powsec}). Say $\left\vert G:G^{q}\right\vert =m$ and $\dim
(G^{q})=(d_{1},\ldots,d_{k}).$ Then $G$ also satisfies%
\[
\mathrm{ind}^{\ast}(\mu_{f,q};m)\wedge\mathrm{res}(\mu_{f,q},\partial
_{d_{1},\ldots,d_{k}}).
\]

We have established that $G$ satisfies%
\[
\sigma_{G}(\overline{a})\equiv\psi(\overline{a})\wedge\mathrm{m}_{f,q}%
\wedge\mathrm{res}(\mu_{f,q},\mathrm{pow}\wedge\mathrm{tf})\wedge
\mathrm{ind}^{\ast}(\mu_{f,q};m)\wedge\mathrm{res}(\mu_{f,q},\partial
_{d_{1},\ldots,d_{k}}).
\]

Now suppose $H\in\mathcal{C}_{\pi}$ satisfies $\sigma_{G}(\overline{h})$. Let
$\theta:G\rightarrow H$ be the epimorphism specified above. To complete the
proof it will suffice to show that $\ker\theta=1.$

$\sigma_{G}(\overline{h})$ implies that $H^{q}=\mu_{f,q}(H)$ is semi-uniform,
that $\left\vert H:H^{q}\right\vert =m=\left\vert G:G^{q}\right\vert $, and
that $\mathrm{\dim}(H^{q})=\mathrm{\dim}(G^{q})$. Applying Lemma \ref{epi} to
$\theta_{\left\vert G^{q}\right.  }$ we infer that $\ker\theta\cap G^{q}=1$.
As $\left\vert H:H^{q}\right\vert =\left\vert G:G^{q}\right\vert $ is finite,
it follows that $\theta$ induces an isomorphism $G/G^{q}\rightarrow H/H^{q}$,
whence $\ker\theta\leq G^{q}$. Thus $\ker\theta=1$ as required.
\end{proof}

\subsection{Finite axiomatizability in profinite groups\label{fapg}}

\label{FApg}

The first case of Theorem \ref{nilp} to be established was for the specific
group $G=\mathrm{UT}_{3}(\mathbb{Z}_{p})$ (\cite{NB}, \S 10). The key point of
the proof is to recover the ring structure of $\mathbb{Z}_{p}$ from the group
structure of $G$: specifically, the commutator map in the group carries enough
information to reconstruct multiplication in the ring. With Theorem
\ref{newT1} at our disposal, we shall see that it will suffice merely to
identify the prime $p$, and this in turn is quite easy provided there are
`enough' commutators. The appropriate condition was identified by Oger and
Sabbagh in \cite{OS}, in the context of abstract nilpotent groups; fortunately
for us it transfers perfectly to the profinite context.

We will say that a group $G$ satisfies the \emph{O-S condition} if
$\mathrm{Z}(G)/(G^{\prime}\cap\mathrm{Z}(G))$ is periodic. Theorem \ref{nilp}
is included in

\begin{theorem}
\label{newT2}Let $L$ be either $L_{\mathrm{gp}}$ or $L_{\pi}$. Suppose that
$G\in\mathcal{C}_{\pi}$ is nilpotent, and that $G$ has an $L$ presentation as
a $\mathcal{C}_{\pi}$ group on the generating tuple $(a_{1},\ldots,a_{r})$.
Then the following are equivalent:

\emph{(a)} $(G,\overline{a})$ is FA (wrt $L$) in the class of all profinite groups;

\emph{(b)} $G$ is FA (wrt $L$) in the class of all f.g.\ nilpotent virtually
pro-$\pi$ groups;

\emph{(c) }$G$ satisfies the O-S condition.
\end{theorem}

For the proof, we write%
\begin{align*}
Z(x)  &  \equiv\forall u.(xu=ux)\\
Z_{2}(x)  &  \equiv\forall v.Z([x,v]);
\end{align*}
these define the centre and second centre in a group. Let $S_{p}(x)$ be a
formula asserting, for $x\in G$, that%
\[
\lbrack Z_{2}(G),x]\subseteq\lbrack Z_{2}(G),x^{p}].
\]
Let $\psi_{\pi}$ be a sentence asserting that $\bigwedge_{p\in\pi}S_{p}(x)
\to[Z_{2}(G),x]=1$ for each $x$.

\begin{lemma}
\label{prop}Let $G$ be a nilpotent profinite group.

\emph{(i)} If $\mathrm{Z}(G)$ is pro-$\pi$ then $G\models\psi_{\pi}.$

\emph{(ii)} If $G\models\psi_{\pi}$ then $G/\mathrm{Z}(G)$ is a pro-$\pi$ group.
\end{lemma}

\begin{proof}
Write $Z=\mathrm{Z}(G)$ and $Z_{2}=\mathrm{Z}_{2}(G)$. We use the facts that
for each $x\in G$ the map $y\longmapsto\lbrack x,y]$ is a continuous
homomorphism from $Z_{2}$ to $Z$, with kernel containing $Z$, and that
$(x,y)\longmapsto\lbrack x,y]$ induces a bilinear map from $G^{\mathrm{ab}%
}\times Z_{2}/Z$ into $Z$.

(i) Suppose that $Z$ is a pro-$\pi$ group. Let $x\in G$. If $S_{p}(x)$ holds
for each $p\in\pi$ then
\[
\lbrack x,Z_{2}]=[x^{p},Z_{2}]=[x,Z_{2}]^{p}~~\text{\ }%
\]
for each $p\in\pi.$ But $[x,Z_{2}]$ is a closed subgroup of the abelian
pro-$\pi$ group\textrm{\ }$Z$, and so $[x,Z_{2}]=1$. Thus $G\models\psi_{\pi}$.

(ii) Let $q\notin\pi$ be a prime, and let $Q$ be a Sylow pro-$q$ subgroup of
$G$. Let $x\in Q$ and $p\in\pi.$ Then $x=x^{\lambda p}$ where $\lambda
\in\mathbb{Z}_{q}$ satisfies $\lambda p=1$. Then for any $u\in Z_{2}$ we have%
\[
\lbrack x,u]=[x^{p},u^{\lambda}].
\]
Thus $S_{p}(x)$ holds.

Now suppose that $G\models\psi_{p}.$ Then $[x,Z_{2}]=1$ holds for each $x\in
Q$, so $Q\cap Z_{2}\leq\mathrm{Z}(Q).$ Now $Q\cap Z_{2}=\mathrm{Z}_{2}(Q)$, so
$\mathrm{Z}_{2}(Q)=\mathrm{Z}(Q)$. As $Q$ is nilpotent this forces
$Q=\mathrm{Z}(Q)\leq Z$.

It follows that $G/Z$ is a pro-$\pi$ group.
\end{proof}

\bigskip

To prove Theorem \ref{newT2}, we have to show that the following are equivalent:

\begin{description}
\item[a] $(G,\overline{a})$ is FA w.r.t. $L$ in the class of all profinite groups;

\item[b] $G$ is FA w.r.t. $L$ in the class of all f.g.\ nilpotent virtually
pro-$\pi$ groups;

\item[c] $G$ satisfies the O-S condition.
\end{description}

Note to begin with that $G$ has finite rank (a familiar property of f.g.
nilpotent groups).

Proposition \ref{neg1} shows that if $G$ does not satisfy the O-S condition
then $G$ cannot be f.a. in the class of groups $\{G\times C_{q},$ $q$
prime$\}$. Thus \textbf{(a)} $\Longrightarrow$ \textbf{(b)} $\Longrightarrow$
\textbf{(c)}, since (\textbf{b}) is formally weaker than (\textbf{a}).

Now suppose that $G$ does satisfy the O-S condition. Given the hypotheses,
Theorem \ref{newT1} gives us a formula $\sigma_{G}$ of $L$ such that for
$b_{1},\ldots,b_{r}\in H\in\mathcal{C}_{\pi},$ we have%
\[
H\models\sigma_{G}(\overline{b})~~\Longleftrightarrow~~(H,\overline{b}%
)\cong(G,\overline{a}).
\]
Also $\mathrm{Z}(G)/(G^{\prime}\cap\mathrm{Z}(G))$ is a periodic pro-$\pi$
group of finite rank, so it is finite, of exponent $q$ say; here $q$ is a
$\pi$-number. Recall (Lemma \ref{dg}) that every element of $G^{\prime}$ is a
product of $d$ commutators, where $d=\mathrm{d}(G)$. Therefore $G$ satisfies%
\begin{equation}
\theta_{q}\equiv\forall y.\left(  Z(y)\longrightarrow\exists u_{1}%
,v_{1},\ldots,u_{d},v_{d}.\left(  y^{q}=\prod_{i=1}^{d}[u_{i},v_{i}]\right)
\right)  . \label{theta}%
\end{equation}

Say $G$ is nilpotent of class $c$. This is expressed by a sentence $\Gamma
_{c}$ (all simple commutators of weight $c+1$ are equal to $1$). Now define%
\[
\Sigma_{G}\equiv\psi_{\pi}\wedge\sigma_{G}\wedge\theta_{q}\wedge\Gamma_{c}.
\]
Then $G$ satisfies $\Sigma_{G}(\overline{a}).$ Suppose $H$ is a profinite
group and that $H\models\Sigma_{G}(\overline{b})$ for some $\overline{b}\in
H^{(r)}$. Then $H$ is nilpotent, so by Proposition \ref{prop}(ii)
$H/\mathrm{Z}(H)$ is a pro-$\pi$ group. Also $\mathrm{Z}(H)H^{\prime
}/H^{\prime}$ has exponent dividing $q$, so $H/H^{\prime}$ is a pro-$\pi$
group. As $H$ is nilpotent this implies that $H$ is a pro-$\pi$ group (to see
this, note that each finite continuous quotient $\widetilde{H}$ of $H$ is the
direct product of its Sylow subgroups, and its abelianization is the direct
product of their respective abelianizations. So if $\widetilde{H}%
/\widetilde{H}^{\prime} $ is a $\pi$-group then the Sylow $q$-subgroups of
$\widetilde{H}$ for $q\notin\pi$ have trivial abelianization, and as they are
nilpotent this means that they are trivial. \ Therefore $\widetilde{H}$ is a
$\pi$-group). Thus $H\in\mathcal{C}_{\pi}$.

As $H\models\sigma_{G}(\overline{b})$ it follows that $(H,\overline{b}%
)\cong(G,\overline{a})$. Thus \textbf{(c)} $\Longrightarrow$ \textbf{(a)}.

\subsection{Finite axiomatizability in pronilpotent groups\label{pronilpsec}}

The nilpotency hypothesis in Theorem \ref{newT2} is very restrictive. Without
it, we can prove a weaker result, giving finite axiomatizability in the class
of all \emph{pronilpotent groups}; this is strictly intermediate between
$\mathcal{C}_{\pi}$ and the class of all profinite groups, so the following
results `interpolate' the two preceding theorems:

\begin{theorem}
\label{T3}Let $G\in\mathcal{C}_{\pi}$ have finite rank, and assume that $G$
has an $L$-presentation in $\mathcal{C}_{\pi}$. Then the following are equivalent:

\emph{(a)} $G$ is FA (wrt $L$) in the class of all pronilpotent groups;

\emph{(b)} $G$ is FA (wrt $L$) in the class of all pronilpotent virtually
pro-$\pi$ groups of finite rank;

\emph{(c) }$G$ satisfies the O-S condition.
\end{theorem}

\begin{theorem}
\label{extrathm}Let $G\in\mathcal{C}_{\pi}$ have finite rank, and assume that
$G$ has an $L$-presentation in $\mathcal{C}_{\pi}$. If $G/\gamma_{m}(G)$
satisfies the O-S condition for some $m\geq2$ then $G$ is FA (wrt $L$) in the
class of all pronilpotent groups.
\end{theorem}

Here, $\gamma_{m}(G)$ denotes the $m$th term of the lower central series (a
closed normal subgroup when $G$ is a f.g.\ profinite group). Note that when
$L$ is $L_{\pi}$, Proposition \ref{L-pres}(ii) makes the assumption of an
$L$-presentation redundant.

We will use the fact that for each $d$ and $m$ there exists $f=f(d,m)$ such
that in any $d$-generator pronilpotent group $G$ we have%
\[
\gamma_{m}(G)=X^{\ast f}:=\{x_{1}\ldots x_{f}\mid x_{1},\ldots,x_{f}\in X\}
\]
where
\[
X=\{[y_{1},\ldots,y_{m}]\mid y_{1},\ldots,y_{m}\in G\}
\]
(\cite{SW}, Lemma 4.3.1). Define a formula%
\[
\Gamma_{m,f}(x)\equiv\exists y_{11},\ldots,y_{fm}.\left(  x=\prod_{i=1}%
^{f}[y_{i1},\ldots,y_{im}]\right)  ;
\]
this asserts that $x$ is a product of $f$ simple left-normed commutators of
weight $m$. For any group $G,$ $\Gamma_{m,f}(G)$ is a subset of $\gamma
_{m}(G)$; and $\Gamma_{m,f}(G)=\gamma_{m}(G)$ if and only if $\Gamma_{m,f}(G)$
is a subgroup, that is, iff%
\[
G\models\mathrm{s}(\Gamma_{m,f})
\]
(see \S \ref{defsub}). In particular, this holds if $G$ is pronilpotent and
$f\geq f(\mathrm{d}(G),m).$

\begin{lemma}
\label{pronilp}Let $H$ be a f.g.\ pronilpotent group. Put $H_{n}=\gamma
_{n}(H)$ and $Z_{n}/H_{n}=\mathrm{Z}(H/H_{n})$ for each $n$. Suppose that
$H_{s}^{q}\leq H_{s+1}$ for some $s\geq1$, where $q$ is a $\pi$-number. Then
$Z_{s}/\mathrm{Z}(H)$ is a pro-$\pi$ group.
\end{lemma}

\begin{proof}
We begin with some properties of the series $(H_{n})$.

\emph{(i) }$H_{n}^{q}\leq H_{n+1}$ for each $n\geq s.$ Proof by induction on
$n$. Let $n\geq s$ and suppose that $H_{n}^{q}\leq H_{n+1}$. Now $H_{n+1}$ is
generated by elements $[x,h]$ with $x\in H_{n}$ and $h\in H$. These satisfy%
\[
\lbrack x,h]^{q}\equiv\lbrack x^{q},h]\equiv1~(\operatorname{mod}H_{n+2}).
\]
As $H_{n+1}/H_{n+2}$ is abelian it follows that $H_{n+1}^{q}\leq H_{n+2}$.

\emph{(ii) }$Z_{n}^{q}\leq Z_{n+1}$ for each $n\geq s.$ To see this, let $z\in
Z_{n}$ and $h\in H$. Then%
\[
\lbrack z^{q},h]\equiv\lbrack z,h]^{q}\equiv1~(\operatorname{mod}H_{n+1})
\]
by (i) so $z^{q}\in Z_{n+1}$.

\emph{(iii)} $\bigcap_{n=s}^{\infty}Z_{n}=\mathrm{Z}(H).$ \ This is immediate
from the fact that $\bigcap_{n=s}^{\infty}H_{n}=1$, which holds because $H$ is pronilpotent.

To conclude the proof, observe that the subgroups $H_{n},$ and therefore also
$Z_{n},$ are closed in $H$. Let $n>s$. Then $H/Z_{n}$ is a f.g.\ nilpotent
profinite group and $Z_{s}/Z_{n}$ has exponent dividing $q^{n-s}$, so
$Z_{s}/Z_{n}$ is a finite $\pi$-group, and $Z_{n}$ is open in $Z_{s}$. The
claim now follows from (iii).
\end{proof}

\begin{lemma}
\label{dcc}Let $G\in\mathcal{C}_{\pi}$ have finite rank, and let
$(G_{n})_{n\in\mathbb{N}}$ be a descending chain of closed normal subgroups of
$G$. Then there exists $s$ such that $G_{n}/G_{n+1}$ is finite for each $n\geq
s$.
\end{lemma}

\begin{proof}
The sequence $\mathrm{Dim}(G/G_{n})$ is non-decreasing and bounded by
$\mathrm{Dim}(G)$, so it becomes stationary at some point $n=s$. Then
$G_{s}/G_{n}$ is finite for all $n\geq s,$ by Corollary \ref{dims}.
\end{proof}

\medskip

Now let $\pi$ be a finite set of primes, let $G\in\mathcal{C}_{\pi}$ have
finite rank, and assume that $G$ has an $L$-presentation in $\mathcal{C}_{\pi
}$. For Theorem \ref{T3} we have to establish the equivalence of

\begin{description}
\item[a] $G$ is FA (wrt $L$) in the class of all pronilpotent groups;

\item[b] $G$ is FA (wrt $L$) in the class of all pronilpotent virtually
pro-$\pi$ groups of finite rank;

\item[c] $G$ satisfies the O-S condition.
\end{description}

Theorem \ref{extrathm} asserts that these follow from

\begin{description}
\item[d] $G/\gamma_{m}(G)$ satisfies the O-S condition for some $m\geq2.$
\end{description}

The proof that \textbf{(a)} $\Longrightarrow$ \textbf{(b)} $\Longrightarrow$
\textbf{(c)} is the same as in the preceding subsection (proof of Theorem
\ref{newT2}).

Fix $d=\mathrm{d}(G).$ By Theorem \ref{newT1}, there is a sentence $\sigma
_{G}$ such that for any $H\in\mathcal{C}_{\pi}$, $H\models\sigma_{G}$ iff
$H\cong G$.

Suppose that \textbf{(d)} holds. We have $\gamma_{m}(G)=\Gamma_{m,f}(G)$ where
$f=f(d,m)$. By Theorem \ref{newT2}, there is a sentence $\Sigma$ such that a
profinite group $L$ satisfies $\Sigma$ iff $L\cong G/\gamma_{m}(G)$.

Let
\[
\xi\equiv\mathrm{s}(\Gamma_{m,f})\wedge\mathrm{lift}(\Gamma_{m,f}%
,\Sigma)\wedge\sigma_{G}%
\]
(note that in any group $H$ satisfying $\mathrm{s}(\Gamma_{m,f})$, the formula
$\Gamma_{m,f}$ defines the normal subgroup $\gamma_{m}(H)$, so $\mathrm{lift}%
(\Gamma_{m,f},\Sigma)$ makes sense). Now suppose that $H$ is a pronilpotent
group and that $H\models\xi$. Then $\Gamma_{m,f}(H)=\gamma_{m}(H)$ is a closed
normal subgroup of $H,$ and $H/\gamma_{m}(H)\cong G/\gamma_{m}(G)$ is a
pro-$\pi$ group. As $m\geq2$ this now implies that $H$ is a pro-$\pi$ group
(as in the proof of Theorem \ref{newT2}, above). Thus $H\in\mathcal{C}_{\pi}$
and so $H\cong G$. Thus $\xi$ determines $G$ among all pronilpotent groups,
and \textbf{(a)} holds.\medskip

Suppose now that \textbf{(c)} holds. According to Lemma \ref{dcc}, there
exists $s\geq2$ such that $\gamma_{s}(G)/\gamma_{s+1}(G)$ is finite. Then
$\gamma_{s}(G)^{q}\leq\gamma_{s+1}(G)$ for some $\pi$-number $q$. Putting
$f=f(d,s)$ and $f^{\prime}=f(d,s+1),$ we see that $G$ satisfies%
\[
\eta\equiv\forall x.\left(  \Gamma_{s,f}(x)\rightarrow\Gamma_{s+1,f^{\prime}%
}(x^{q})\right)  .
\]
Condition \textbf{(c)} implies that $G$ satisfies $\theta_{q^{\prime}},$
defined in (\ref{theta}), for some $\pi$-number $q^{\prime}$. Let $\psi_{\pi}$
be as in Lemma \ref{prop}. Then $G/\gamma_{s}(G)\models\psi_{\pi}$ so
$G\models\mathrm{lift}(\Gamma_{s,f},\psi_{\pi})$.

Now put%
\[
\Sigma\equiv\mathrm{s}(\Gamma_{s,f})\wedge\eta\wedge\theta_{q^{\prime}}%
\wedge\mathrm{lift}(\Gamma_{s,f},\psi_{\pi})\wedge\sigma_{G}.
\]
Let $H$ be a pronilpotent group and define $Z_{n}\geq H_{n}=\gamma_{n}(H)$ as
in Lemma \ref{pronilp}. Suppose that $H$ satisfies $\Sigma$. Then
$H_{s}=\Gamma_{s,f}(H)$ and $H/H_{s}\models\psi_{\pi}$. It follows by Lemma
\ref{prop} that $H/Z_{s}$ is a pro-$\pi$ group.

As $H\models\eta$ we have $H_{s}^{q}\subseteq\Gamma_{s+1,f^{\prime}%
}(H)\subseteq H_{s+1}$, so $Z_{s}/\mathrm{Z}(H)$ is a pro-$\pi$ group, by
Lemma \ref{pronilp}. As $H\models\theta_{q^{\prime}}$, we have $\mathrm{Z}%
(H)^{q^{\prime}}\leq H^{\prime}$. It follows that $H/H^{\prime}$ is a
pro-$\pi$ group, and hence (as before) that $H$ is pro-$\pi$. As
$H\models\sigma_{G}$ it follows that $H\cong G$.

Thus $\Sigma$ determines $G$, and so \textbf{(c)} implies \textbf{(a)}.

\section{Special linear groups\label{slnsec}}

Here we show how Theorem \ref{slsl} may be deduced, quite neatly, from
Theorems \ref{newT2} and \ref{newT1}. The result for $\mathrm{SL}%
_{n}(\mathbb{Z}_{p})$ with $n\geq3$ (and all $p$) was subsequently established
in \cite{ST} using bi-interpretability; the present relatively simple
group-theoretic proof may be adaptable to other algebraic groups, where
bi-interpretability is not known.

For the proof of Theorem \ref{slsl}, we start with the congruence subgroup%
\[
G=\mathrm{SL}_{n}^{1}(\mathbb{Z}_{p})=\ker\left(  \mathrm{SL}_{n}%
(\mathbb{Z}_{p})\rightarrow\mathrm{SL}_{n}(\mathbb{F}_{p})\right)  .
\]
Until further notice we assume that $n\geq3.$ We also assume for convenience
that $p\nmid2n.$

Note that $G$ is a uniform pro-$p$ group; this is easy to see directly, or
from \cite{DDMS}\textbf{, }Theorem 5.2 and Lemma \ref{add-dim} above .

Set $u_{ij}=1+pe_{ij},$ $v_{ij}=1-pe_{ji}$ for the elementary upper and lower
triangular matrices ($1\leq i<j\leq n$). The $u_{ij}$ (respectively $v_{ij}$)
form a basis for the full upper (respectively lower) unitriangular group $U$
(resp. $V$) in $G$, each of which is a nilpotent uniform pro-$p$ group that
satisfies the O-S condition (Subsection \ref{fapg}).

Put $T_{p}=1+p\mathbb{Z}_{p}$. Let $\eta\in T_{p}\setminus T_{p}^{p}$ satisfy
$\eta^{2p}=(1+p^{2})^{-1}$, and $\zeta_{i}\in T_{p}$ satisfy $\zeta^{n}%
\eta^{n-2i}=1.$ Set%
\[
h_{i}=\zeta_{i}1_{n}\cdot\mathrm{diag}(\eta^{-1},\ldots,\eta^{-1},\eta
,\ldots,\eta)~\text{for}~i=1,\ldots,n-1,
\]
where the last $\eta^{-1}$ occurs in the $i$th place and the first $\eta$ in
the $(i+1)$th place. Then $h_{1},\ldots,h_{n-1}$ form a basis for the diagonal
group $H$ in $G$, which is a free abelian pro-$p$ group. Note that%
\begin{equation}
G=V\cdot H\cdot U. \label{gvhu}%
\end{equation}

We take $\overline{u}=(u_{ij}),~\overline{v}=(v_{ij})$ and $\overline
{h}=(h_{i})$ as parameters.

\begin{theorem}
\label{sln_1thm copy(1)}$(G,(\overline{u},\overline{v},\overline{h}))$ is FA
in the class of all profinite groups.
\end{theorem}

\begin{proof}
For brevity, we will say `formula' to mean `formula of $L_{\mathrm{gp}}$ with
parameters $\overline{u},\overline{v},\overline{h}$', except where parameters
are explicitly mentioned. We make the following \emph{claims}.

\begin{enumerate}
\item[(\textbf{1})] $H$ is definable, in fact $H=\chi(\overline{h};G)$ for
some formula $\chi$, such that $\chi$ always defines a closed subgroup in any
profinite group. Indeed, $H=\bigcap_{1}^{n-1}\mathrm{C}_{G}(h_{i})$.

\item[(\textbf{2})] $U$ and $V$ are definable, in fact $U=\varphi
_{1}(\overline{u},\overline{h};G),~V= \varphi_{2}(\overline{v},\overline
{h};G)$ for some formulae $\varphi_{1},~\varphi_{2}$, that always define
closed subsets in any profinite group. This follows from the fact that $U$ is
the product (in a suitable order) of the `elementary' subgroups $U_{ij}%
=\overline{\left\langle u_{ij}\right\rangle }=\mathrm{Z}(\mathrm{C}_{G}%
(u_{ij})),$ and similarly for $V$.

\item[(\textbf{3})] $U$ has an $L_{\mathrm{gp}}$ presentation on $\overline
{u}$ as a pro-$p$ group; $V$ has an $L_{\mathrm{gp}}$ presentation on
$\overline{v}$ as a pro-$p$ group. This follows from \textbf{\cite{DDMS}},
Proposition 4.32, and the usual commutator relations.

\item[(\textbf{4})] $G$ has an $L_{\mathrm{gp}}$ presentation on
$(\overline{u},\overline{v},\overline{h})$ as a pro-$p$ group. This is not
quite immediate. The usual commutator relations provide a pro-$p$ presentation
for the uniform pro-$p$ group $G$. However, some of them cannot be expressed
in $L_{\mathrm{gp}}$ as they involve non-integral powers, and a roundabout
argument is required; this is indicated below.
\end{enumerate}

Given these claims, the proof is concluded as follows.

By Claims \textbf{1} and \textbf{2}, we can construct a formula $\Phi
(\overline{x},\overline{y},\overline{z})$ such that $\Phi(\overline
{u},\overline{v},\overline{h})$ expresses the conjunction of the facts (which
are true in $G$):

\emph{(a)} multiplication maps $V\times H\times U$ bijectively to $G$,

\emph{(b) }$[V,H]\leq V^{p}$ and $\mathrm{C}_{H}(V)=1$,

\emph{(c) }$[U,H]\leq U^{p}$ and $\mathrm{C}_{H}(U)=1,$

\emph{(d)} $H$ is abelian and has no $p$-torsion.

As $U$ and $V$ are nilpotent pro-$p$ groups satisfying the O-S condition, and
given Claim \textbf{3}, Theorem \ref{newT2} provides formulae $\sigma
_{U}(\overline{x})$, $\sigma_{V}(\overline{x})$ that determine $(U,\overline
{u})$ and $(V,\overline{v})$ among all profinite groups. Let $\Psi
(\overline{x},\overline{y},\overline{z})$ express the conjunction of the following:

\emph{(e) \ }$\mathrm{s}(\varphi_{1}(\overline{x},\overline{z}))\wedge
\mathrm{s}(\varphi_{2}(\overline{y},\overline{z}))$, i.e. $\varphi
_{1}(\overline{x},\overline{z})$ and $\varphi_{2}(\overline{x},\overline{z})$
define subgroups,

\emph{(f) \ }$\mathrm{res}(\varphi_{1}(\overline{x},\overline{z}),\sigma
_{U}(\overline{x})),$ i.e. the subgroup defined by $\varphi_{1}(\overline
{x},\overline{z})$ satisfies $\sigma_{U}(\overline{x}),$

\emph{(g) \ }$\mathrm{res}(\varphi_{2}(\overline{y},\overline{z}),\sigma
_{V}(\overline{y})),$ i.e. the subgroup defined by $\varphi_{2}(\overline
{y},\overline{z})$ satisfies $\sigma_{V}(\overline{y}).$

\medskip Now suppose that $\widetilde{G}$ is a profinite group and
$\overline{u}^{\thicksim},\overline{v}^{\thicksim},\overline{h}^{\thicksim}$
are tuples in $\widetilde{G}$ of the appropriate lengths such that
\[
\widetilde{G}\models\Phi(\overline{u}^{\thicksim},\overline{v}^{\thicksim
},\overline{h}^{\thicksim})\wedge\Psi(\overline{u}^{\thicksim},\overline
{v}^{\thicksim},\overline{h}^{\thicksim}).
\]
Let $\widetilde{U},~\widetilde{V},~\widetilde{H}$ be the subsets of
$\widetilde{G}$ defined by $\varphi_{1}(\overline{u}^{\thicksim},\overline
{h}^{\thicksim}),~\varphi_{2}(\overline{v}^{\thicksim},\overline{h}%
^{\thicksim}),~\chi(\overline{h}^{\thicksim})$. Then $\Psi$ ensures that
$\widetilde{U}\cong U$ and $\widetilde{V}\cong V$ are pro-$p$ groups,
generated respectively by $\overline{u}^{\thicksim},~\overline{v}^{\thicksim}%
$. Also $\Phi$ ensures that $\widetilde{H}$ is closed, normalizes
$\widetilde{V}$, acting faithfully by conjugation, and that $[\widetilde{V}%
,\widetilde{H}]\subseteq\widetilde{V}^{p}$; this now implies that
$\widetilde{H}$ is a pro-$p$ group, and hence that $\widetilde{V}%
\cdot\widetilde{H}$ is a pro-$p$ group. $\Phi$ also ensures that
$\widetilde{G}=\widetilde{V}\cdot\widetilde{H}\cdot\widetilde{U}$. Since a
product of two pro-$p$ subgroups is again pro-$p$, \ it now follows
that $\widetilde{G}$ is a pro-$p$ group.

Finally, Claim \textbf{4 }with Theorem \ref{newT1} provides a formula
$\sigma_{G}(\overline{x},\overline{y},\overline{z})$ that determines
$(G,\overline{u},\overline{v},\overline{h})$ among pro-$p$ groups. It follows
that
\[
\Phi(\overline{x},\overline{y},\overline{z})\wedge\Psi(\overline{x}%
,\overline{y},\overline{z})\wedge\sigma_{G}(\overline{x},\overline
{y},\overline{z})
\]
determines $(G,\overline{u},\overline{v},\overline{h})$ among all profinite groups.

\medskip

\emph{Proof of Claim}\textbf{ 4.} \textbf{\cite{DDMS}}, Proposition 4.32 show
that $G$ has a presentation as a pro-$p$ group with relations that express the
commutators of basis elements $v_{ij},~u_{ij},~h_{k}$ in canonical form, as
products of $p$-adic powers in a fixed order (see the proof of Lemma
\ref{sfpres}). \ In most cases, this canonical form is either $1$ or $x^{\pm
p}$ where $x$ is a basis element; let us call the corresponding set of
relations $\mathcal{R}_{1}$. The exceptions are those which express
$[v_{ij},u_{ij}],$ $[u_{ij},h_{k}],$ $[v_{ij},h_{k}],$ where the canonical
form involves non-integral powers of basis elements; let us call this set of
relations $\mathcal{R}_{2}$.

Relations in $\mathcal{R}_{2}$ of the first kind can be re-written as
\begin{equation}
\lbrack v_{ij},u_{ij}]=h_{i-1}^{p}h_{i}^{-p}v_{ij}^{-p^{2}}u_{ij}^{p^{2}%
}h_{j-1}^{-p}h_{j}^{p}. \label{niceone}%
\end{equation}

For the others, one verifies that%
\begin{align}
u_{ij}^{h_{k}^{-p}}  &  =u_{ij}^{1+p^{2}}~~(i\leq k<j)\label{nice(uvh)}\\
v_{ij}^{h_{k}^{p}}  &  =v_{ij}^{1+p^{2}}~~(i\leq k<j). \label{niceuvh2}%
\end{align}

Now we have shown above that given Claims 1, 2 and 3, there are formulae
$\Phi$ and $\Psi$ such that for any profinite group $\widetilde{G}$, if
$\widetilde{G}\models\Phi(\overline{u}^{\thicksim},\overline{v}^{\thicksim
},\overline{h}^{\thicksim})\wedge\Psi(\overline{u}^{\thicksim},\overline
{v}^{\thicksim},\overline{h}^{\thicksim})$ then $\widetilde{H}$ is abelian,
has no $p$-torsion, and acts faithfully on $\widetilde{U}$. In this situation,
the action of $h\in\widetilde{H}$ on $\widetilde{U}$ is determined by the
action of $h^{p}$, and similarly for the action on $\widetilde{V}$. This now
implies that (\ref{nice(uvh)}), (\ref{niceuvh2}) are equivalent to the
`canonical' relations specifying $[u_{ij},h_{k}]$ and $[v_{ij},h_{k}]$.

Let $\Delta(\overline{u},\overline{v},\overline{h})$ be a formula that
expresses the relations $\mathcal{R}_{1}$ together with (\ref{niceone}),
(\ref{nice(uvh)}) and (\ref{niceuvh2}). The preceding argument shows that
$\Delta\wedge\Phi\wedge\Psi$ is equivalent to the conjunction of $\Phi
\wedge\Psi$ with the original set of relations $\mathcal{R}_{1}\cup
\mathcal{R}_{2}$. As the latter give a pro-$p$ presentation of $G$, it follows
that $\Delta\wedge\Phi\wedge\Psi$ is an $L_{\mathrm{gp}}$ presentation of $G$
as a pro-$p$ group.
\end{proof}

\bigskip

\textbf{The case} $n=2$. \ In the above argument, the hypothesis $n\geq3$ is
only essential to ensure that $U$ and $V$ satisfy the O-S condition, which in
turn is only used to establish that $\widetilde{U}\cong U$ and $\widetilde{V}%
\cong V$ are pro-$p$ groups. If $n=2$, this has to be established by a
different route. The idea is to show that the ring $\mathbb{Z}_{p}$ can be
interpreted in $G$ by the definable subgroup $U$, and then use the fact that
$\mathbb{Z}_{p}$ is FA in the class of rings whose additive group is profinite
(see \S \ref{profringsec} and the proof of Theorem \ref{sl2_biint copy(1)}).
Together, these allow us to express the fact that $U\cong\mathbb{Z}_{p}$ by a
suitable formula with parameters $u,h.$ The same applies with $V,v$ in place
of $U,u$. With these formulae in place of $\mathrm{res}(\varphi_{1}%
(\overline{x},\overline{z}),\sigma_{U}(\overline{x}))$ and $\mathrm{res}%
(\varphi_{2}(\overline{y},\overline{z}),\sigma_{V}(\overline{y}))$ in the
definition of $\Psi(\overline{x},\overline{y},\overline{z})$, one finds that
$\widetilde{U}\cong\widetilde{V}\cong\mathbb{Z}_{p}$, and the argument then
proceeds as before.

\medskip

Finally, we go from $G$ to the full linear group; here $n$ may be any integer
$\geq2$.

Set $u_{0}=1+e_{12},$ and let $w=e_{12}+\cdots+e_{n-1,n}\pm e_{n,1}$ be the
permutation matrix for the $n$-cycle $(12\ldots n),$ adjusted to have
determinant equal to $1$. We need two easy lemmas:

\begin{lemma}
Let $m\in\mathbb{Z}$ with $m\equiv1~(\operatorname{mod}p)$. There is a formula
$\chi_{m}(\overline{y},\overline{z},x)$ such that%
\[
G\models\chi_{m}(\overline{v},\overline{h},x)\Longleftrightarrow
x=h(m^{-1}):=\mathrm{diag}(m,m^{-1},1,\ldots,1).
\]

\end{lemma}

\begin{lemma}
Put $m=1+p$. Then%
\[
u_{0}^{-1}v_{12}u_{0}=v_{12}^{m^{-1}}h(m^{-1})u_{12}^{m^{-1}},~~\ (u_{0}%
^{-1}h_{1}u_{0})^{p}=u_{12}^{p}h_{1}^{p}.
\]

\end{lemma}

Now we can deduce

\begin{theorem}
The groups $\mathrm{SL}_{n}(\mathbb{Z}_{p})$ and $\mathrm{PSL}_{n}%
(\mathbb{Z}_{p})$ are FA in the class of all profinite groups.
\end{theorem}

\begin{proof}
Write $\symbol{126}:\Gamma=\mathrm{SL}_{n}(\mathbb{Z}_{p})\rightarrow
\widetilde{\Gamma}=\mathrm{PSL}_{n}(\mathbb{Z}_{p})$ for the quotient map. As
$\symbol{126}$ restricts to an injective map on $G$, we may consider both
$\Gamma$ and $\widetilde{\Gamma}$ as finite extensions of $G$. As
$\mathrm{Z}(G)=1$ we may apply Theorem \textbf{\ref{FinExt}}(ii).\textbf{\ }%
The following argument deals with $\Gamma$; the same argument with
$\symbol{126}$ applied to everything will give the result for
$\widetilde{\Gamma}$.

For convenience, we shall allow $\overline{u},\overline{v},\overline{h}$ to
denote the \emph{sets} $\{u_{ij}\ldots\}$ etc., as well ordered tuples. By
conjugating with $w$ and forming commutators and inverses we can obtain every
elementary matrix from $u_{0}$. It follows that
\begin{align*}
G=\overline{\left\langle \overline{u},\overline{v},\overline{h}\right\rangle
}\vartriangleleft &  \Gamma=\overline{\left\langle u_{0},w,\overline
{h}\right\rangle },\\
\left\langle \overline{u},\overline{v},\overline{h}\right\rangle  &
\subseteq\left\langle u_{0},w,\overline{h}\right\rangle .
\end{align*}

Thus it will suffice to verify that $x^{y}$ is $(\overline{u},\overline
{v},\overline{h})$-definable in $G$ for each $x\in\overline{u}\cup\overline
{v}\cup\overline{h}$ and $y\in\{u_{0},w,\overline{h}\}$. This is obvious for
$y\in\overline{h};$ for $y=w$ it follows from the relations $h_{i}^{w}%
=h_{i}^{-1}h_{i+1}h_{i-1}^{w}.$

Now $u_{0}$ commutes with every $u_{ij}$ and every $h_{k}$ for $k\geq2$, and
conjugates each $v_{ij}$ with $(i,j)\neq(1,2)$ into the group $\left\langle
\overline{u},\overline{v}\right\rangle $. So it remains to deal with
$v_{12}^{u_{0}}$ and $h_{1}^{u_{0}}$.

We use the two preceding lemmas, and keep their notation. Set $m=1+p,$%
\begin{align*}
\alpha(\overline{u},\overline{v},\overline{h},x,y,z)  &  :=(x^{m}%
=v_{12})\wedge\chi_{m}(\overline{v},\overline{h},y)\wedge(z^{m}=u_{12}),\\
\varphi(\overline{u},\overline{v},\overline{h},t)  &  :=\exists x,y,z.\left(
t=xyz\wedge\alpha(\overline{u},\overline{v},\overline{h},x,y,z)\right)  .
\end{align*}
Then $G\models\varphi(\overline{u},\overline{v},\overline{h},v_{12}^{u_{0}})$,
and this defines $v_{12}^{u_{0}}$ in $G,$ using a `uniqueness of expression'
property in the identity (\ref{gvhu}).

Put%
\[
\psi(\overline{u},\overline{h},t):=(t^{p}=u_{12}^{p}h_{1}^{p}).
\]
Then $\psi(\overline{u},\overline{h},t)$ defines $h_{1}^{u_{0}}$ in $G,$
because extraction of $p$th roots is unique in the uniform pro-$p$ group $G$.

The result follows from Theorem \ref{FinExt}(ii).
\end{proof}

\section{Some negative results \label{negsec}}

\subsection{Infinitely many primes}

The Feferman-Vaught theorem (\cite{CK}, Proposition 6.3.2) implies the
following: for each sentence $\phi$ of $L_{\mathrm{gp}}$ there exist finitely
many sentences $\psi_{1},\ldots,\psi_{n}$ of $L_{\mathrm{gp}}$ and a formula
$\theta(x_{1},\ldots,x_{n})$ in the language of Boolean algebras such that for
any family of groups $\{G_{i}\mid i\in I\},$ setting $X_{j}=\{i\in I\mid
G_{i}\models\psi_{j}\}$ we have%
\[
G(I):=\prod_{i\in I}G_{i}\models\phi\Longleftrightarrow\mathcal{P}%
(I)\models\theta(X_{1},\ldots,X_{n})
\]
($\mathcal{P}(I)$ denotes the power set of $I$). To deduce Proposition
\ref{products}, note that by the pigeonhole principle, provided $\left\vert
I\right\vert >2^{n}$ there exist $r\neq q\in I$ such that the characteristic
functions of $X_{1},\ldots,X_{n}$ each take the same values on $r$ and $q$.
Then for each $j$ we have $G_{r}\models\psi_{j}\Longleftrightarrow
G_{q}\models\psi_{j},$ so setting $H_{i}=G_{i}$ ($i\neq q$) and $H_{q}=G_{r}$
we see that $H(I)\models\phi\Longleftrightarrow G(I)\models\phi$. The
proposition follows.

Now we may rephrase Proposition \ref{negprod} as follows:

\begin{proposition}
Let $\pi$ be an infinite set of primes.

\emph{(i)} If $G_{p}\neq1$ is a pro-$p$ group for each $p\in\pi$ then
$\prod_{p\in\pi}G_{p}$ is not FA in the class of pronilpotent groups.

\emph{(ii)} Let $\mathfrak{G}$ be a linear algebraic group of positive
dimension defined over $\mathbb{Q}$.
Then $\mathfrak{G}(\mathbb{Z}_{\pi})$ is not FA in the class of profinite groups.
\end{proposition}

\begin{proof}
(i) Suppose $\phi$ is a sentence satisfied by $G:=\prod_{p\in\pi}G_{p}$. For
some $r\neq q\in\pi$ we have%
\[
H:=G_{r}\times\prod_{p\in\pi\smallsetminus\{q\}}G_{p}\models\phi.
\]
But $H=H^{q}$ and $G\neq G^{q},$ so $H\ncong G$ and so $\phi$ does not
determine $G$.

(ii) Similarly, if $\phi$ determines $\mathfrak{G}(\mathbb{Z}_{\pi})=$
$\prod_{p\in\pi}\mathfrak{G}(\mathbb{Z}_{p})$ among profinite groups it
follows that for some $r\neq q\in\pi$,%
\begin{equation}
H:=\mathfrak{G}(\mathbb{Z}_{r})\times\mathfrak{G}(\mathbb{Z}_{\pi
\smallsetminus\{q\}})\cong\mathfrak{G}(\mathbb{Z}_{\pi})\label{isomo}%
\end{equation}
as profinite groups. The projection $\mathfrak{G}(\mathbb{Z}_{\pi}%
)\rightarrow\mathfrak{G}(\mathbb{Z}_{q})$ then induces a surjective continuous
homomorphism $\mu:H\rightarrow\mathfrak{G}(\mathbb{Z}_{q}).$

Write $\mathfrak{G}^{1}(\mathbb{Z}_{p})$ for the congruence subgroup modulo
$p$ (or modulo $4$ if $p=2$). This is a pro-$p$ group, it is torsion-free and
has finite index in $\mathfrak{G}(\mathbb{Z}_{p})$ (see \cite{DDMS}, \S 5.1).
Hence for $p\in\pi\smallsetminus\{q\}$ we have $\mathfrak{G}^{1}%
(\mathbb{Z}_{p})\mu\cap\mathfrak{G}^{1}(\mathbb{Z}_{q})=1$, and it follows
that $\mathfrak{G}(\mathbb{Z}_{p})\mu$ is finite, as well as normal in
$\mathfrak{G}(\mathbb{Z}_{q})$ (here $\mathfrak{G}(\mathbb{Z}_{r})$ is
appearing twice, as it does in $H$).

Thus writing $K$ for the maximal finite normal subgroup of $\mathfrak{G}%
(\mathbb{Z}_{q})$ we deduce:%
\[
\mathfrak{G}(\mathbb{Z}_{p})\mu\leq K \text{ for all }p\in\pi\smallsetminus
\{q\}
\]
(this applies in particular to both copies of $\mathfrak{G}(\mathbb{Z}_{r})$
in $H$). As the subgroups $\mathfrak{G}(\mathbb{Z}_{p})$ generate $H$
topologically this implies that $\mathfrak{G}(\mathbb{Z}_{q})=K$ is finite.
This is not the case, because $\mathfrak{G}(\mathbb{Z}_{q})$ is open in
$\mathfrak{G}(\mathbb{Q}_{q})$, a $q$-adic manifold of positive dimension
$\mathrm{dim}(\mathfrak{G})$ (see \cite{PR}, \S 3.1).

\end{proof}

\medskip\textbf{Remark.} If the reductive part of $\mathfrak{G}$ is
semisimple, (\ref{isomo}) cannot even be an abstract group isomorphism; this
requires a little more argument. When $\mathfrak{G}$ is connected, semisimple
and simply connected, the profinite group $H$ is finitely generated
(\cite{PR2}, Thm. 2), so $\mu$ is continuous by Theorem \ref{NSSC}. In
general, one can show that the image in $H$ of the simply-connected cover of
the connected component of $\mathfrak{G}$ contains $H^{m}$ for some $m$; the
argument then goes as before, noting that $\mathfrak{G}(\mathbb{Z}_{q})^{m}$
has finite index in $\mathfrak{G}(\mathbb{Z}_{q})$. We are grateful to Andrei
Rapinchuk for help with this question.

\subsection{Uncountably many pro-$p$ groups\label{manysec}}

For $\lambda\in\mathbb{Z}_{p}$ let $T(\lambda)$ be the class-$2$ nilpotent
pro-$p$ group with pro-$p$ presentation on generators $x_{1},\ldots
,x_{4},y_{1},\ldots,y_{4},e,f$ and relations%
\begin{align*}
\lbrack x_{i},x_{j}]  &  =[y_{i},y_{j}]=1\text{ \ (all }i,j\text{)}\\
\lbrack x_{i},y_{j}]  &  =1\text{ \ \ \ \ \ \ \ \ \ \ \ \ (all }i\neq
j\text{)}\\
\lbrack x_{1},y_{1}]  &  =e,~~~~[x_{2},y_{2}]=ef^{-1}\\
\lbrack x_{3},y_{3}]  &  =f,~~~~[x_{4},y_{4}]=ef^{-\lambda}\\
&  e,~f\text{ central}%
\end{align*}
This is clearly a pro-$p$ group of rank $10$ (with centre $\mathbb{Z}_{p}^{2}
$ and central quotient $\mathbb{Z}_{p}^{8}$) and so contains an open normal
uniform subgroup $T^{\ast}(\lambda)$ (for example the subgroup generated by
$e,~f$ and the $p$th powers of the $x_{i}$ and $y_{j}$). It is proved in
\cite{GS} \S 6, page 153, that the groups $T(\lambda)$ are pairwise
non-commensurable. It follows that the groups $T^{\ast}(\lambda)$ are pairwise non-isomorphic.

Note that $T(\lambda)$ is strictly f.p. when $\lambda$ is a rational $p$-adic
integer $a/b$ ($a,b\in\mathbb{Z}$, $p\nmid b$): the relation involving
$\lambda$ is equivalent to
\[
\lbrack y_{4},x_{4}]^{b}e^{b}=f^{a},~~\ \ [x_{4},y_{4}]~\text{central}%
\]
(because we have unique extraction of $b$th roots).

\section{List of formulae}

\label{s:listform}

\noindent$\mathrm{s}(\kappa),~~\mathrm{s}_{\vartriangleleft}(\kappa
),~\ \mathrm{res}(\kappa,\varphi),~~\mathrm{lift}(\kappa,\varphi
),~\ \mathrm{ind}(\kappa;n),~\ \mathrm{ind}^{\ast}(\kappa;n),~~\mathrm{com}%
(x,y):\qquad$\S \ref{defsub} \medskip\newline$\delta(\overline{a}%
,x),~\ \beta_{d}(a_{1},\ldots,a_{d}),~\ \widetilde{\beta}_{d},~~\beta
_{d}^{\ast}:\qquad$\S \ref{prelimsec} \medskip\newline$\mathrm{pow},$
\ $\mu_{f,q}(x),~~\mathrm{m}_{f,q},~~\partial_{d_{1},\ldots,d_{k}}:\qquad
$\S \ref{powsec} \medskip\newline$\mathrm{tf}:$\qquad\S \ref{proofT1}
\medskip\newline$Z(x),~~Z_{2}(x),~~S_{p}(x),~~\psi_{\pi},~~\Gamma_{c}:$%
\qquad\S \ref{FApg} \medskip\newline$\Gamma_{m,f}:$\qquad\S \ref{pronilpsec}

\bigskip

\emph{Authors' addresses} \medskip

Prof. Andr\'{e} Nies \ \texttt{$<$andre@cs.auckland.ac.nz$>$}

School of Computer Science

University of Auckland

Private Bag 92019

Auckland 1142

New Zealand

\bigskip

Katrin Tent \ \texttt{$<$tent@wwu.de$>$}

Mathematisches Institut

Universit\"{a}t M\"{u}nster

Einsteinstrasse 62

48149 M\"{u}nster

Germany

\bigskip

Dan Segal \ \texttt{$<$dan.segal@all-souls.ox.ac.uk$>$}

All Souls College

Oxford OX1 4AL

GB

\end{document}